\documentclass[dvipdf]{amsart}

\title[]{Extremal length functions are log-plurisubharmonic}
\author{Hideki Miyachi}
\address{Department of Mathematics,
Graduate School of Science,
Osaka University,
Machikaneyama 1-1, Toyonaka, Osaka 560-0043, Japan}

\makeatletter
\@addtoreset{equation}{section}

\makeatother

\usepackage{amsfonts, amsmath, amssymb, amsthm}
\usepackage{url}
\usepackage[dvipdfm]{graphicx}
\usepackage[dvipdfm]{color}
\usepackage{amscd}

\newtheorem{theorem}{Theorem}[section]
\newtheorem{lemma}{Lemma}[section]

\newtheorem{corollary}{Corollary}[section]

\newtheorem{remark}{Remark}[section]

\newtheorem{convention}{Convention}

\newcommand{\ext}{{\rm Ext}}

\newcommand{\teich}[1]{\mathcal{T}_{#1}}
\newcommand{\quaddiff}[1]{\mathcal{Q}_{#1}}

\newcommand{\ThurstonSymplectic}[3]{\omega_{Th}\left( #1\,,\,#2\right)_{#3}}

\newcommand{\horizontal}{\mathbf{h}}
\newcommand{\vertical}{\mathbf{v}}

\newcommand{\slice}{E^{v}}

\newcommand{\complexstructure}{\mathcal{J}}

\newcommand{\douadyhubbard}{{\rm DH}}

\newcommand{\levi}[1]{{\mathcal L}(#1)}

\newcommand{\Mod}{{\rm Mod}}

\newcommand{\Leb}[1]{\frac{i}{2}d{#1}\wedge d\overline{#1}}

\begin{document}
\maketitle

\begin{abstract}
In this paper,
we show that the extremal length functions on Teichm\"uller space
are log-plurisubharmonic.
As a corollary,
we obtain an alternative proof of L.Liu and W.Su's results
on the plurisubharmonicity of extremal length functions.
We also obtain alternative proofs of S.Krushkal's results that
a function defined by the Teichm\"uller distance
from a reference point
is plurisubharmonic,
and the Teichm\"uller space is hyperconvex.
To show the log-plurisubharmonicity,
we give an explicit formula of the Levi form of the extremal length functions
in generic case.
\end{abstract}

\tableofcontents

\section{Introduction}
\subsection{Results}
A positive function $U$ on a complex manifold $N$ is said to be
\emph{log-plurisubharmonic} if
$\log U$ is plurisubharmonic.
Any log-plurisubharmonic function is plurisubharmonic
(cf. \S\ref{subsec:log-plurisubharmonicity}).

Let $\teich{g,m}$ be the Teichm\"uller space of Riemann surfaces
of analytically finite type $(g,m)$ with $2g-2+m>0$.
One of the aims of this paper is to give
the following theorem,
which improves
a result by L.Liu and W. Su
in \cite{2012arXiv1210.0743L}.

\begin{theorem}[Log-plurisubharminicity]
\label{thm:pluri-subharmonicity}
Extremal length functions on $\teich{g,m}$
are log-plurisubharmonic.
\end{theorem}
Moreover,
we will also observe that
for generic measured foliations,
the log-extremal length functions are real analytic strictly plurisubharmonic functions on $\teich{g,m}$
(cf. Theorem \ref{thm:strong-positivity2-2}).

When $3g-3+m=1$,
Theorem \ref{thm:pluri-subharmonicity} follows
from the direct calculation
(cf. \S\ref{subsec:examples}).
Hence,
in the discussion in the later section,
we always assume that $3g-3+m\ge 2$.
We will prove Theorem \ref{thm:pluri-subharmonicity}
in \S\ref{subsec:log-plurisubharmonicity}.
In fact,
Theorem \ref{thm:pluri-subharmonicity} is derived from
a more stronger property of the extremal length functions
than that given in Theorem \ref{thm:pluri-subharmonicity}
(cf. Theorem \ref{thm:plurisuperharmonicitiy-of-reciprocals}).
However,
the property in Theorem \ref{thm:pluri-subharmonicity}
is also important as we see below
(cf. Corollaries \ref{coro:plurisubharmoncity-of-positive-polynomial}
and \ref{coro:Teichmuller distance}).
%
To show Theorem \ref{thm:pluri-subharmonicity},
we will give an explicit formula of the Levi forms of
extremal length functions
in generic case.
(cf. Theorem \ref{thm:leviform}).

The following immediately follows from Theorem \ref{thm:pluri-subharmonicity} (cf. \cite[Colollary 2.6.9]{MR1150978}).

\begin{corollary}[Log-plurisubharmonicity of positive polynomials]
\label{coro:plurisubharmoncity-of-positive-polynomial}
Let $F_{1},\cdots,F_{n}$ be measured foliations.
For any polynomial $P$ of $n$-variables with positive coefficients,
the function
$$
\teich{g,m}\ni x\mapsto
P(
\ext_{x}(F_{1}),\cdots,\ext_{x}(F_{n})
)
\in \mathbb{R}
$$
is log-plurisubharmonic.
\end{corollary}

\subsection{Teichm\"uller distance is plurisubharmonic}
%
%
It is known that the least upper semicontinuous majorant of 
a family of plurisubharmonic funtions is either a constant function
$+\infty$ or plurisubharmonic
(cf. Theorem 5 in \cite[Chapter II, \S2]{MR0243112}).
Kerckhoff's formula \eqref{eq:Kerckhof_formula}
implies that the Teichm\"uller distance $d_{T}(x_{0},\cdot)$ from a reference point $x_{0}\in \teich{g,m}$
coincides with the half of the least upper semicontinuous majorant
of the family of log-extremal length functions.
Therefore,
we obtain an alternative proof of Krushkal's result
\cite[Corollary 3]{MR1142683}
as follows.

\begin{corollary}[Teichm\"uller distance is plurisubharmonic]
\label{coro:Teichmuller distance}
For $x_{0}\in \teich{g,m}$,
the distance function
$$
\teich{g,m}\ni x\mapsto d_{T}(x_{0},x)
$$
is plurisubharmonic.
\end{corollary}

\subsection{Convexities for Teichm\"uller space}
A subset $K$ in a complex manifold $N$ is said to be
\emph{disk-convex} in $N$
if for any continuous mapping $g\colon \overline{\mathbb{D}}\to N$,
holomorphic in $\mathbb{D}$,
$f(\partial \mathbb{D})\subset K$ implies
$f(\overline{\mathbb{D}})\subset K$
(cf. \cite{MR1478844}).
For $\epsilon>0$ and a measured foliation $F$,
we call the set of the form $\{x\in \teich{g,m}\mid \ext_{x}(F)<\epsilon\}$
the \emph{$\epsilon$-horoball for $F$}.
From the maximum modulus principle
for (pluri)subharmonic functions,
we obtain the following.

\begin{corollary}[Disk convexity]
\label{coro:diskconvexity}
For $\epsilon>0$ and a measured foliation $F$,
the $\epsilon$-horoball for $F$ is disk-convex  in $\teich{g,m}$.
The metric ball with respect to the Teichm\"uller distance
is also disk-convex  in $\teich{g,m}$.
\end{corollary}

About the convexity of metric balls,
A. Lenzhen and K. Rafi \cite{MR2838267} observed that
the metric ball with respect to the Teichm\"uller distance
is quasi-convex.
Furthermore,
they also showed that the extremal length function is not convex along a Teichm\"uller geodesic
in general.
This means that the intersection between a horoball of extremal length and a Teichm\"uller disk
is simply connected,
but may not be hyperbolically convex in general.

A complex manifold $N$ is said to be \emph{hyperconvex}
(in the sense of Stehl\'e)
if it admits a continuous plurisubharmonic exhaustion function
$\rho\colon N\to [-\infty,0)$
in the sense that
$\{x\in N\mid \rho(x)<c\}$ is relatively compact in $M$ for every $c<0$ (cf. \cite{MR0374493}. See also \cite{MR1210002}).
Take two essentially complete measured foliations $F$ and $G$
which are transverse in the sense that
\begin{equation}
\label{eq:transverse-mf}
i(F,H)+i(G,H)>0
\end{equation}
for all $H\in \mathcal{MF}-\{0\}$.
In Theorem \ref{thm:plurisuperharmonicitiy-of-reciprocals}
given later,
we will check that
\begin{equation}
\label{eq:extremal-length-exhausion}
\rho\colon \teich{g,m}\ni x\mapsto -\frac{1}{\ext_{x}(F)+\ext_{x}(G)+1}
\end{equation}
is a real analytic, negative, strictly plurisubharmonic exhaustion function
with lower bounds.
Thus we obtain the following.
%
\begin{corollary}[Hyperconvexitiy]
\label{coro:hyperconvexity}
Teichm\"uller space is hyperconvex.
\end{corollary}
Corollary \ref{coro:hyperconvexity}
is already given by Krushkal in \cite{MR1119946}.
Notice that the hyperconvexity of Teichm\"uller space
also follows from the compleness of the Carath\'eodory distance
(cf. \cite{MR534932} and \cite{MR0352450}).
As an immediate corollary of the hyperconvexity
(or Corollary \ref{coro:Teichmuller distance}),
by virtue of Oka's theorem,
we deduce the following,
which is first proven by L.Bers and L.Ehrenpreis
(see \cite{MR0168800}.
See also \cite{MR0352450},
\cite{2012arXiv1210.0743L},
\cite{MR766636}, 
\cite{MR880186}).

\begin{corollary}[Holomorphic convexity]
\label{coro:stein}
Teichm\"uller space is
a domain of holomorphy.
\end{corollary}

\section{Notation}

\subsection{Teichm\"uller space}
Let $\Sigma_{g,m}$ be a compact orientable surface of genus $g$
with $m$-disks removed with $2g-2+m>0$.
The \emph{Teichm\"uller space} $\teich{g,m}$ of
Riemann surfaces of analytically finite type $(g,m)$
is the set of equivalence classes of
pairs $(M,f)$ consisting of a Riemann surface $M$ of analytically finite type $(g,m)$
and an orientation preserving homeomorphism $f\colon {\rm Int}(\Sigma_{g,m})\to M$.
Two marked Riemann surfaces $(M_{1},f_{1})$ and $(M_{2},f_{2})$
are \emph{equivalent} if
there is a conformal mapping $h\colon M_{1}\to M_{2}$ which homotopic to $f_{2}\circ f_{1}^{-1}$.

The Teichm\"uller space admits a canonical distance $d_{T}$,
which we call the \emph{Teichm\"uller distance}.
The Teichm\"uller distance is originally 
defined as the logarithm of 
the infimum of the maximal dilatations
of quasiconformal mappings respecting the markings
(cf. \cite[\S5]{MR1215481}).
S. Kerckhoff \cite{MR559474}
gave a geometric description of the Teichm\"uller distance
via the extremal length,
which we will recall in \eqref{eq:Kerckhof_formula} below.

\subsection{Measured foliations}
\label{subsec:measured-foliatons}
Let $\mathcal{S}$ be the set of homotopy classes
of non-trivial and non-peripheral simple closed curves on $\Sigma_{g,m}$.
Let $\mathcal{WS}$ be the set of weighted simple closed curves,
that is, the set of formal products
$t\alpha$ of non-negative number $t$ and $\alpha\in \mathcal{S}$.
The closure $\mathcal{MF}$ of the embedding
$$
\mathcal{WS}\ni t\alpha
\mapsto [\mathcal{S}\ni \beta\mapsto t\,i(\alpha,\beta)]\in
\mathbb{R}_{+}^{\mathcal{S}}
$$
is called the \emph{space of measured foliations} on $\Sigma_{g,m}$.
We consider $\mathcal{WS}$ as
a subset of $\mathcal{MF}$.
We identify $1\cdot \alpha\in \mathcal{WS}$
with $\alpha\in \mathcal{S}$.
Any measured foliation is described as a pair consisting
of a singular foliation and a transverse measure
(cf. \cite[Expos\'e 5]{MR568308}).
For instance,
$t\alpha\in \mathcal{MF}$ is a foliated annulus with core $\alpha$
which identified with an annulus
$[0,t]\times [0,1]/(x,0)\sim (x,1)$
and a transverse measure associated to $|dx|$.

It is known that
the geometric intersection number $i(\alpha,\beta)$
for $\alpha,\beta\in \mathcal{S}$
extends continuously to $\mathcal{MF}\times \mathcal{MF}$
(cf. \cite{MR847953}).

If we fix a complete hyperbolic structure on the interior of $\Sigma_{g,m}$,
any measured foliation canonically corresponds to a measured lamination.
A \emph{measured lamination} is a pair of a geodesic lamination,
which called the \emph{support},
and a transverse measure,
where a \emph{geodesic lamination} a  is a compact set 
in the interior fo $\Sigma_{g,m}$ which consists
of disjoint complete geodesics
(cf. \cite[\S1.7]{MR1144770}).
A measured foliation is said to be \emph{essentially complete}
if the support of the corresponding measured lamination is maximal,
that is,
the complement consists of ideal triangles or ideal punctured monogon
if $(g,m)\ne (1,1)$,
a punctured bigon otherwise
(cf. \cite[Definition 9.5.1, Propositions 9.5.2 and 9.5.4]{Thuston-LectureNote}).
In the interior of $\Sigma_{g,m}$,
each singularity of the associated foliation of an essentially complete measured foliation is a three prong singularity.
At any puncture,
the associated foliation has a one prong singularity
(cf. \cite[Epilogue]{MR1144770}).

\subsection{Extremal length}
\label{subsec:extremal_length}
Let $M$ be a Riemann surface and let $A$ be a doubly connected domain
on $M$.
If $A$ is conformally equivalent to a round annulus $\{1<|z|<R\}$,
we define the \emph{modulus} of $A$ by
$$
\Mod(A)=\frac{1}{2\pi}\log R.
$$
The \emph{extremal length} of
a simple closed curve $\alpha$
on $M$ is defined by
\begin{equation}
\label{eq:gemetric_definition_extremal_length}
\ext_M(\alpha)=\inf\left\{\frac{1}{\Mod(A)}\mid
\mbox{the core curve of $A\subset M$ is homotopic to $\alpha$}
\right\}.
\end{equation}
In \cite{MR559474},
Kerckhoff showed that 
if we define the extremal length of $t\alpha\in \mathcal{WS}$
by
$$
\ext_M(t\alpha)=t^2\ext_M(\alpha),
$$
then the extremal length function $\ext_M$ on $\mathcal{WS}$
extends continuously to $\mathcal{MF}$.
For $F\in \mathcal{MF}$,
the extremal length function $\ext_{\cdot}(F)$
on $\teich{g,m}$ is defined by
$$
\ext_x(F)=\ext_M(f(F)).
$$
It is known that the function
$$
\teich{g,m}\times \mathcal{MF}\ni (x,F)\mapsto \ext_{x}(F)
$$
is continuous.
Furthermore,
the extremal length function $\ext_{x}\colon\mathcal{MF}\to \mathbb{R}$
is a proper function and satisfies the quasiconformal distortion property:
$$
\ext_{y}(F)\le e^{2d_{T}(x,y)}\ext_{x}(F)
$$
for any $F\in \mathcal{MF}$ and $x,y\in \teich{g,m}$.

\subsection{Kerckhoff's formula and Minsky's inequality}
\label{Kerckhoff-Minsky}
The Teichm\"uller distance $d_T$ is expressed with the extremal length,
which we call \emph{Kerckhoff's formula}:
\begin{align}
\label{eq:Kerckhof_formula}
d_T(x_1,x_2)&=\frac{1}{2}\log \sup_{\alpha\in \mathcal{S}}
\frac{\ext_{x_2}(\alpha)}{\ext_{x_1}(\alpha)}
\\
&=
\frac{1}{2}\sup_{\alpha\in \mathcal{S}}
\left(
\log\ext_{x_2}(\alpha)-\log\ext_{x_1}(\alpha)
\right)
\nonumber
\end{align}
for $x_{1},x_{2}\in \teich{g,m}$
(see \cite{MR559474}).
Y. Minsky \cite{MR1231968} observed the following inequality,
which we recently call \emph{Minsky's inequality}:
\begin{equation}
\label{eq:Minsky_inequality}
i(F,G)^2\le \ext_x(F)\,\ext_x(G)
\end{equation}
for $x\in \teich{g,m}$ and $F,G\in \mathcal{MF}$.
If two measured foliation $F$ and $G$ are transverse
in the sense of \eqref{eq:transverse-mf},
we have
$$
\ext_{x}(F)+\ext_{x}(G)\ge\frac{i(F,H)^{2}+i(G,H)^{2}}{\ext_{x}(H)}
$$
for all $H$ and hence we deduce
$$
\ext_{x}(F)+\ext_{x}(G)\ge e^{2d_{T}(x_{0},x)}
\min\{i(F,H)^{2}+i(G,H)^{2}\mid \ext_{x_{0}}(H)=1\}.
$$
This implies that the function $x\mapsto \ext_{x}(F)+\ext_{x}(G)$
is proper on $\teich{g,m}$.
%


\subsection{Quadratic differentials}
Let $\quaddiff{g,m}$ be the space of holomorphic quadratic differentials
with finite $L^{1}$-norm.
Namely,
the space $\quaddiff{g,m}$ is a holomorphic vector bundle over $\teich{g,m}$
and the fiber $\quaddiff{x}$ at $x=(M,f)\in \teich{g,m}$ is
the space of holomorphic quadratic differentials on $M$
with finite $L^{1}$-norm
(cf. \S\ref{subsec:infinitesimal-theory}).
Any $q\in \quaddiff{x}$ extends meromorphically to
the completion $\overline{M}$ by filling punctures
and has (at most) simple poles at punctures.
A quadratic differential $q\in \quaddiff{g,m}$ is said to be \emph{generic}
if it satisfies the following two conditions:
\begin{enumerate}
\item
$q$ has a simple pole at every puncture.
\item
all zeroes of $q$ are simple.
\end{enumerate}
We can easily check that the set 
of all generic holomorphic quadratic differentials is an open and dense subset in 
$\quaddiff{g,m}$.

\subsection{Hubbard-Masur differentials}
For $q\in \quaddiff{g,m}$,
there is a unique measured foliation $F$ such that
$$
i(F,\alpha)=\inf_{\alpha'\sim f(\alpha)}\int_{\alpha'}|{\rm Re}\sqrt{q}|.
$$
for all $\alpha\in \mathcal{S}$,
where $(M,f)=\pi(q)\in \teich{g,m}$.
We call $F$ the \emph{vertical foliation} and denote it by
$\vertical(q)$.
We call $\horizontal(q)=\vertical(-q)$
the \emph{horizontal foliation} of $q$.
%
%
It is known that for any $F\in \mathcal{MF}$
and $x=(M,f)\in \teich{g,m}$,
there is a unique holomorphic quadratic differential
$q_{F,x}$ such that
$\vertical(q_{F,x})=F$.
We call $q_{F,x}$ the \emph{Hubbard-Masur differential}
for $F$ on $x$
(cf. \cite{MR523212}).
If $F$ is essentially complete,
the Hubbard-Masur differential $q_{F,x}$ is generic for all $x\in \teich{g,m}$,
since any Whitehead equivalent measured foliation to $F$
is isotopic to $F$.

The Hubbard-Masur differential $q_{F,x}=q_{F,x}(z)dz^2$
for $F$ on $x=(M,f)$ satisfies
\begin{equation}
\label{eq:JS_extremallength_norm_length}
\ext_x(F)=\|q_{F,x}\|=\int_M|q_{F,x}|=\int_M|q_{F,x}(z)|
\cdot\Leb{z}.
\end{equation}


\subsection{Infinitesimal theory}
\label{subsec:infinitesimal-theory}
Teichm\"uller space also has a canonical complex structure,
which induced from the deformation by quasiconformal mappings.
A \emph{Beltrami differential}
is,
by definition,
an $L^{\infty}$-measurable $(-1,1)$-form.
The holomorphic tangent space
is described as the quotient of the complex Banach space
of the Beltami differentials
(cf. \cite[Theorem 7.5]{MR1215481}).
The holomorphic cotangent space is identified with the space
of holomorphic quadratic differentials.
The natural pairing between holomorphic tangent space
and cotangent space is given by
$$
\langle v,q\rangle=\int_{M}\dot{\mu}q=
\int_{M}\dot{\mu}(z)q(z)\cdot\Leb{z}
$$
for $v=[\dot{\mu}]\in T_{x}\teich{g,m}$
and $q\in T^{*}_{x}\teich{g,m}\cong \quaddiff{x}$
and $x=(M,f)\in \teich{g,m}$.

\begin{convention}
Henceforth,
for $(p_{i},q_{i})$ forms $\varphi_{i}=\varphi_{i}(z)dz^{p_{i}}d\overline{z}^{q_{i}}$ $(i=1,2)$ on a Riemann surface $M$
with $p_{1}+p_{2}=q_{1}+q_{2}=1$,
we write
$$
\int_{M}\varphi_{1}\varphi_{2}=
\int_{M}\varphi_{1}(z)\varphi_{2}(z)\cdot\Leb{z}.
$$
\end{convention}

\section{Coordinates via representations of the odd cohomology}
%

\subsection{The tangent spaces to quadratic differentials}
\label{subsec:tangent-space-to-quadratic-differentials}
Following \cite{MR523212},
we describe the tangent space $T_{q_{0}}\quaddiff{g,m}$
as the first hypercohomolgy group $\mathbb{H}^{1}(L^{\bullet})$
a complex of sheaves
(cf. \cite{MR0102797} or \cite{MR507725}).
We will need the Kodaira-Spencer identification of the
tangent space of Teichm\"uller space
with the first cohomology group of the sheaf of holomorphic vector fields
(for instance, see \cite{MR815922}. See also \cite{MR1215481} and \cite{MR499279}).

Let $X$ and $q$
be a holomorphic vector field 
and a holomorphic quadratic differential
on an open set of a Riemann surface $M$.
Denote by $L_{X}q$ the Lie derivative of $q$ along $X$.
Let $\Theta_{M}$ and $\Omega_{M}^{\otimes 2}$ be the sheaves of germs of
holomorphic vector fields with zeroes at punctures
and meromorphic quadratic differentials on $M$ with (at most)
first order poles at punctures,
respectively.
Let $q_{0}\in \quaddiff{g,m}$
($q_{0}$ need not to be generic).
The tangent space $T_{q_{0}}\quaddiff{g,m}$
is identified with the 
first hypercohomology group of the complex of sheaves
$$
\begin{CD}
L^{\bullet}\colon\quad
0 @>>> \Theta_{M} @>{L_{\cdot}q_{0}}>> \Omega_{M}^{\otimes 2}@>>>0
\end{CD}
$$
(cf. \cite[Proposition 4.5]{MR523212}).
For the convenience,
we recall the definition of the (first)
hypercohomology group which we use here.
The first cochain group is the direct sum
$C^{0}(M,\Omega_{M}^{\otimes 2})\oplus C^{1}(M,\Theta_{M})$.
Consider an appropriate covering $\mathcal{U}=\{U_{i}\}_{i}$ on $M$
such that $\mathbb{H}^{1}(L^{\bullet})\cong \mathbb{H}^{1}(\mathcal{U},L^{\bullet})$
(see the proof of \cite[Proposition 4.5]{MR523212}).
A cochain $(\{\phi_{i}\}_{i},\{X_{ij}\}_{i,j})$
in $C^{0}(\mathcal{U},\Omega_{M}^{\otimes 2})\oplus
C^{1}(\mathcal{U},\Theta_{M})$
is said to be \emph{cocycle} if
it satisfies
$$
\delta\{X_{ij}\}_{i,j}=X_{ij}+X_{jk}+X_{ki}=0,\quad
\delta\{\phi_{i}\}_{i}=\phi_{i}-\phi_{j}=L_{X_{ij}}(q_{0}).
$$
A \emph{coboundary} is a cochain $(\{\phi_{i}\}_{i},\{X_{ij}\}_{i,j})$
of the form
$$
X_{ij}=Z_{i}-Z_{j}=\delta\{Z_{i}\}_{i},\quad
\phi_{i}=L_{Z_{i}}(q_{0})
$$
for some $0$-cochain $\{Z_{i}\}_{i}\in C^{0}(\mathcal{U},\Theta_{M})$
(cf. Figure \ref{fig:hypercohomology}).
\begin{figure}[t]
$$
\begin{CD}
0
\\
@AAA \\
C^{0}(\mathcal{U},\Omega_{M}^{\otimes 2})
@>{\delta}>>
C^{1}(\mathcal{U},\Omega_{M}^{\otimes 2}) \\
@A{L_{\cdot}q_{0}}AA
@A{-L_{\cdot}q_{0}}AA
\\
C^{0}(\mathcal{U},\Theta_{M})
@>{\delta}>>
C^{1}(\mathcal{U},\Theta_{M})
@>{\delta}>>
C^{2}(\mathcal{U},\Theta_{M})
\end{CD}
$$
%
\caption{Double complex for the tangent spaces to $\quaddiff{g,m}$}
\label{fig:hypercohomology}
\end{figure}

\subsection{Branched covering space associated with generic differentials}
\label{subsec:Branched-Covering-Space-associated-with-generic-differentials}
Let $x=(M,f)\in \teich{g,m}$ and $q_{0}$ a generic holomorphic
quadratic differential on $M$.
We consider $M$ as a pair of topological surface $\Sigma_{g}$
and a complex strtucture on $\Sigma_{g}-\{\mbox{punctures of $M$}\}$.

Let $U$ be a contractible neighborhood of $q_{0}$ in $\quaddiff{g,m}^{0}$.
For $q\in U$,
let $A_{q}={\rm Zero}(q)\cup {\rm Pole}(q)\subset \Sigma_{g}$
and $M_{q}$ the underlying surface of $q$.
Then,
we have a homomorphism
\begin{equation}
\label{eq:representation-q}
\rho_{q}\colon \pi_{1}(\Sigma_{g}-A_{q})\to \mathbb{Z}/2\mathbb{Z}
\subset {\rm Isom}(\mathbb{C})
\end{equation}
associated to the double covering space
$$
\pi_{q}\colon \tilde{M}_{q}\to M_{q}
$$
of Riemann surface of $\sqrt{q}$
over $\Sigma_{g}-A_{q}$,
where $\mathbb{Z}/2\mathbb{Z}$ in \eqref{eq:representation-q}
is recognized as
a subgroup of the isometry group ${\rm Isom}(\mathbb{C})$
of $\mathbb{C}$ with the standard Euclidean metric
generated by the $\pi$-rotation.
The square root $\sqrt{q}$ lifts to a well-defined holomorphic $1$-
form $\omega_{q}$ on $\tilde{M}_{q}$.

Let $W=\Sigma_{g}-A_{q_{0}}$
and $W'$ be a compact subsurface of $W$
with smooth boundary
such that
the inclusion $W'\hookrightarrow W_{0}$
is homotopy equivalent and
the complement $\Sigma_{g}-W'$
is a union of tiny open $|q_{0}|$-disks of centers in $A_{0}$
By taking sufficiently small $U$ if necessary,
we may assume that $A_{q}\subset  \Sigma_{g}-W'$
for any $q\in U$.
Hence,
by passing the inclusions $W'\hookrightarrow W$
and $W'\hookrightarrow \Sigma_{g}-A_{q}$,
we get a canonical identification between
$\pi_{1}(W)$ and $\pi_{1}(\Sigma_{g}-A_{q})$
so that $\rho_{q}=\rho_{0}$ as a homomorphism from
$\pi_{1}(W)$ to $\mathbb{Z}/2\mathbb{Z}$
for all $q\in U$.

\subsection{Backgrounds}
\label{subsec:background-tangent-space}
Let $q_{0}\in \quaddiff{g,m}$ be a generic differential on $M_{q_{0}}$.
Take a neighborhood $U$ of $q_{0}$ as the previous section.
Then, $\mathbb{M}=\cup_{q\in U}\tilde{M}_{q}$ is a holomorphic family of closed Riemann surfaces over $U$,
and the arrangement of the covering transformation $r_{q}\colon
\tilde{M}_{q}\to \tilde{M}_{q}$ defines the holomorphic automorphism
${\bf r}$ of  $\mathbb{M}$ of order $2$.
Taking a sufficiently small $U$ if necessary,
we may assume that there is a covering $\{\mathcal{V}_{i}\}_{i}$
of $\mathbb{M}$ such that
\begin{enumerate}
\item
each $\mathcal{V}_{i}$ admits a biholomorphic mapping $\mathcal{A}_{i}$
onto the product domain $\mathbb{D}\times U$
such that its restriction to $\mathcal{V}_{i,q}=\mathcal{V}_{i}\cap \tilde{M}_{q}$ is a biholomorphism onto $\mathbb{D}\times \{q\}$
for each $q\in U$,
\item
for any finite $i_{1},\cdots,i_{k}$ and $q\in U$,
the intersection $\mathcal{V}_{i_{1},q}\cap \cdots\cap \mathcal{V}_{i_{k},q}$
is connected and simply connected,
\item
eacn $V_{i,q}$ contains at most one branch point of $\tilde{M}_{q}\to M_{q}$
and any branch point of $\tilde{M}_{q}\to M_{q}$ is contained in at most
one open set in $\{V_{i,q}\}_{i}$,
and
\item
the covering $\{\mathcal{V}_{i}\}_{i}$ is equivariant under ${\bf r}$
in the sense that 
$\{{\bf r}(\mathcal{V}_{i})\}_{i}=\{\mathcal{V}_{i}\}_{i}$
and $\mathcal{A}_{j}\circ {\bf r}=\mathcal{A}_{i}$
when ${\bf r}(\mathcal{V}_{i})=\mathcal{V}_{j}$.
\end{enumerate}
Set $D=\{\lambda=s+it\mid |s|,|t|<\delta_{0}\}$.
Let $\{q_{\lambda}\}_{\lambda\in D}$
be a differentiable family of holomorphic quadratic differenitals
with $q_{0}=q_{0}$ and $q_{\lambda}\in U$ for $\lambda\in D$.
For any $\lambda\in D$,
we let $\alpha_{i}(\lambda)\colon \mathcal{V}_{i,q_{\lambda}}\to \mathbb{D}$
by $\alpha_{i}(\lambda)=\mathcal{A}_{i}\mid_{\tilde{M}_{q_{\lambda}}}$
and set
$V_{\lambda;ij}=\alpha_{j}(\lambda)(\mathcal{V}_{i,q_{\lambda}}\cap \mathcal{V}_{j,q_{\lambda}})$,
and
$\alpha_{ij}(\lambda)=\alpha_{j}(\lambda)\circ \alpha_{i}(\lambda)^{-1}\colon V_{\lambda;ji}\to V_{\lambda;ij}$.
For simplicity,
set $\alpha_{i}=\alpha_{0;i}$,
$\alpha_{ij}=\alpha_{0;ij}$
and $V_{ij}=V_{0;ij}$.
Consider the image of $\alpha_{j}$ is in the $z_{j}$-plane.

Let $\tilde{q}_{\lambda;i}dz_{i}^{2}$ be the representation of $(\pi_{q_{\lambda}})^{*}(q_{\lambda})$
under the coordinate $\alpha_{i}(\lambda)$.
Let
\begin{align}
&
\tilde{\phi}_{\lambda;i}=\frac{\partial \tilde{q}_{\lambda;i}}{\partial \lambda},
\quad
\tilde{X}_{\lambda;i}=-\frac{\partial \alpha_{i}(\lambda)}{\partial \lambda}\circ
\alpha_{i}(\lambda)^{-1},
\label{eq:vector-fields}
\end{align}
on $\mathbb{D}$
for all $i$ and set
$$
\tilde{X}_{\lambda;ij}(z_{j})=\frac{\partial \alpha_{ij}(\lambda)}{\partial \lambda}(z_{i})
,
$$
where $z_{j}=\alpha_{ij}(\lambda)(z_{i})$ and $z_{i}\in V_{\lambda;ij}$
(tildes mean objects (differentials or vector fields etc.)
on the covering space $\tilde{M}_{q_{0}}$
which obtained as lifts of objects on $M_{q_{0}}$).
Then,
$$
\tilde{X}_{\lambda;i}(z_{i})\frac{\partial}{\partial z_{i}},\quad
\tilde{X}_{\lambda;ij}(z_{j})\frac{\partial}{\partial z_{j}}
$$
are vector fields on appropriate sets of $\tilde{M}_{q_{\lambda}}$.
Set $\tilde{q}_{i}=\tilde{q}_{0;i}$
and
$\tilde{\phi}_{i}=\tilde{\phi}_{0;i}$.
Notice that
\begin{equation}
\label{eq:subsec:background-tangent-space1}
\alpha_{ij}(\lambda)_{*}(\tilde{X}_{\lambda;i})-\tilde{X}_{\lambda;j}=\tilde{X}_{\lambda;ij},
\end{equation}
\begin{align}
\alpha_{ij}(\lambda)^{*}(\tilde{\phi}_{\lambda;i})
-\tilde{\phi}_{\lambda;j}
&=L_{\tilde{X}_{\lambda;ij}}(\tilde{q}_{\lambda;j})
\label{eq:subsec:background-tangent-space2}
\end{align}
on $V_{\lambda;ij}$.
We abbreviate these equations to
\begin{align}
&\tilde{X}_{\lambda;i}-\tilde{X}_{\lambda;j}=\tilde{X}_{\lambda;ij}
,\quad
\tilde{\phi}_{\lambda;i}
-\tilde{\phi}_{\lambda;j}
=L_{\tilde{X}_{\lambda;ij}}(\tilde{q}_{\lambda}).
\label{eq:subsec:background-tangent-space4}
\end{align}
for instance.
Notice again that $\omega_{q_{\lambda}}^{2}=(\pi_{q_{\lambda}})^{*}(q_{\lambda})$
for all $\lambda\in D$.

Notice that for any cocycle $\{X_{ij}\}_{ij}\in C^{1}(M,\Theta_{M})$,
we can always find a $0$-cochain
$\{X_{i}\}_{i}$ of the sheaf
of $C^{\infty}$-vector fields which satisfies
$X_{i}-X_{j}=X_{ij}$ on $U_{i}\cap U_{j}$
applying a partition of unity.
Then,
from \eqref{eq:vector-fields},
the $\overline{z}$-derivative
\begin{equation}
\label{eq:infinitesimal1}
\dot{\nu}=-(\tilde{X}_{i})_{\overline{z}}
\end{equation}
defines a Beltrami differential on $M$
which
represents the infinitesimal deformation
corresponding
to the tangent vector
associated to
the cohomology class of $\{X_{ij}\}_{i,j}$ in $H^{1}(M,\Theta_{M})$.
Compare with Equation $(7.27)$ in \cite[\S7.2.4]{MR1215481}.

\subsection{Local coordinates via homomorphisms}
Let $H_{1}(\tilde{M}_{q})^{-}$ is the odd part of the homology
of $\tilde{M}_{q}$ with coefficient in $\mathbb{Z}$.
Namely,
$H_{1}(\tilde{M}_{q})^{-}$ is the the eigen space of the eigen value $-1$
of the linear automorphism $(r_{q})_{*}$ on the first cohomology group 
of $\tilde{M}_{q}$ defined by the covering transformation $r_{q}$.
In \cite{MR0396936},
Douady and Hubbard consider the following mapping
\begin{equation}
\label{eq:Douady-Hubbard}
\douadyhubbard\colon
U\ni \lambda \mapsto \chi_{q}\in {\rm Hom}(H_{1}(\tilde{M}_{q_{0}})^{-},\mathbb{C})
\end{equation}
defined by
$$
\chi_{q}(c)=\int_{c}\omega_{q}
$$
where we identify $H_{1}(\tilde{M}_{q_{0}})^{-}\cong H_{1}(\tilde{M}_{q})^{-}$
in a canonical way as discussed in \S\ref{subsec:Branched-Covering-Space-associated-with-generic-differentials}.
Notice that
${\rm Hom}(H_{1}(\tilde{M}_{q_{0}})^{-},\mathbb{C})$ is a $\mathbb{C}$-vector space
of complex dimension $6g-6+2m$.
We shall check the following.

\begin{lemma}[Differential of $\douadyhubbard$]
\label{lem:Douady-Hubbard}
Let $q_{0}\in \quaddiff{g,m}$ be a generic differential.
Let $v\in T_{q_{0}}\quaddiff{g,m}$ be a tangent vector
associated to a $1$-cochain $(\{\phi_{i}\}_{i},\{X_{ij}\}_{i,j})$.
Let $\{X_{i}\}_{i}$ be a $0$-cochain of the sheaf
of $C^{\infty}$-vector fields on $X_{q_{0}}$.
Then,
the derivative of the mapping \eqref{eq:Douady-Hubbard} along $v$
satisfies
\begin{equation}
\label{eq:lem:Douady-Hubbard1}
\douadyhubbard_{*}[v](c)
=\int_{c}
\Phi
\end{equation}
for all $c\in H_{1}(\tilde{M}_{q_{0}})^{-}$,
where
$$
\Phi=\left(
\frac{\tilde{\phi}_{i}}{2\omega_{q_{0}}}-\omega_{q_{0}}'\tilde{X}_{i}-\omega_{q_{0}}(\tilde{X}_{i})_{z}
\right)dz
-\omega_{q_{0}}
(\tilde{X}_{i})_{\overline{z}}d\overline{z}
$$
on $\mathcal{V}_{i,0}$.
\end{lemma}
One can easily check that $\Phi$ in Lemma \ref{lem:Douady-Hubbard}
is a well-defined closed one-form on $\tilde{M}_{q_{0}}$
which satisfies $(r_{q_{0}})^{*}(\Phi)=-\Phi$.

\begin{proof}
We calculate the $\lambda$-derivative at $\lambda=0$
with the notation in \S\ref{subsec:background-tangent-space}.
Notice that
\begin{align*}
\alpha_{i}(\lambda)\circ \alpha_{i}^{-1}(z_{i})
&=z_{i}-\lambda \tilde{X}_{i}(z_{i})-\overline{\lambda}\tilde{Y}_{i}(z_{i})+o(|\lambda|) \\
\tilde{q}_{\lambda;i}(z_{i})
&=\tilde{q}_{i}(z_{i})+\lambda \tilde{\phi}_{i}(z_{i})+
\overline{\lambda}\tilde{\psi}_{i}(z_{i})+o(|\lambda|)
\end{align*}
on $\mathbb{D}=\alpha_{i}(\mathcal{V}_{i,0})$
as $\lambda\to 0$,
where
$$
\tilde{Y}_{i}=-\left.\frac{\partial \alpha_{i}(\lambda)}{\partial \overline{\lambda}}\right|_{\lambda=0}\circ \alpha_{i}^{-1},\quad
\tilde{\psi}_{i}=\left.
\frac{\partial q_{\lambda;i}}{\partial \overline{\lambda}}
\right|_{\lambda=0},
$$
and the convergence is valid uniformly on any compact sets
of $\mathbb{D}$.
In our case,
$\lambda$-derivative of $D\ni\lambda\mapsto \douadyhubbard(q_{\lambda})(c)$
at $\lambda=0$
will be the desired formula.

Notice that the partial $\lambda$-
derivative of $\omega_{\lambda}$ at $\lambda=0$
is
$$
\left.
\frac{\partial \omega_{\lambda}}{\partial \lambda}
\right|_{\lambda=0}
=\frac{\tilde{\phi}_{i}}{2\omega_{q_{0}}}
$$
on $\mathbb{D}=\alpha_{i}(\mathcal{V}_{i,0})$.
Notice also in general that the derivative of
a function defined by the integration
$$
(-\delta,\delta)\ni s\mapsto \int_{a(s)}^{b(s)}f(s,x)dx
$$
at $s=0$
is equal to
$$
\int_{a(0)}^{b(0)}\frac{\partial f}{\partial s}(0,x)dx
+f(0,b(0))\frac{db}{ds}(0)-f(0,a(0))\frac{da}{ds}(0).
$$
Hence,
by the standard argument,
like as the discussion in Proposition 1 of \cite{MR0396936},
one can check the equations \eqref{eq:lem:Douady-Hubbard1}
hold.
\end{proof}

Though the following is well-known
(cf. \cite{MR866707},
\cite{MR1094714},
\cite{MR1135877}),
we confirm for the following for the completeness.
Indeed,
the calculation in the following proof will be a guide
in the later discussion.

\begin{lemma}[$\douadyhubbard$ defines a local coordinate]
\label{lem:douady-hubbard-local-coordinate}
Let $q_{0}\in \quaddiff{g,m}$ be as above.
When we take the above neighborhood $U$ of $q_{0}$
to be sufficiently small,
the mapping \eqref{eq:Douady-Hubbard} defines a holomorphic local coordinate around $q_{0}$.
\end{lemma}

\begin{proof}
%
Let $v\in T_{q_{0}}\quaddiff{g,m}$ be a tangent vector
which is represented 
by the $1$ cocycle
$(\{\phi_{i}\}_{i},\{X_{ij}\}_{i,j})$
in 
$C^{0}(\mathcal{V}_{0},\Omega_{M_{q_{0}}}^{\otimes 2})\oplus
C^{1}(\mathcal{V}_{0},\Theta_{M_{q_{0}}})$,
where $\mathcal{V}_{0}=\{\{\pi_{q_{0}}(\mathcal{V}_{i,0})\}_{i}\}_{i}$
is a covering of $M_{q_{0}}$
(cf. \eqref{eq:subsec:background-tangent-space2}).

Suppose that
$\douadyhubbard_{*}[v]=0$.
Let $\Phi$ be a closed one form defined in Lemma \ref{lem:Douady-Hubbard}.
Since $(r_{0})^{*}(\Phi)=-\Phi$,
$$
\displaystyle \int_{c}\Phi=
\displaystyle \int_{r_{0}(c)}\Phi
=-\displaystyle \int_{c}\Phi
$$
for all $c\in H_{1}(\tilde{M}_{q_{0}})^{+}$,
where $H_{1}(\tilde{M}_{q_{0}})^{+}$ is the even part of the homology
of $\tilde{M}_{q}$ with coefficient in $\mathbb{Z}$.
Hence,
there is a smooth function $f$ on $\tilde{M}_{q_{0}}$ such that
$df=\Phi$
from de Rham's theorem.
Namely,
\begin{align*}
\frac{\tilde{\phi}_{i}}{2\omega_{q_{0}}}-\omega_{q_{0}}'\tilde{X}_{i}-
\omega_{q_{0}}\cdot (\tilde{X}_{i})_{z}
&=f_{z} \\
-\omega_{q_{0}}\cdot (\tilde{X}_{i})_{\overline{z}}
&=
f_{\overline{z}}
\end{align*}
hold on $\mathbb{D}=\alpha_{i}(\mathcal{V}_{i,0})$,
where $\tilde{\phi}_{i}$ and
$\tilde{X}_{i}$ are lifts of $\phi_{i}$ and $\tilde{X}_{i}$ to $\tilde{M}_{q_{0}}$,
respectively.
Since $\{(\mathcal{V}_{i,\lambda},\alpha_{i}(\lambda))\}_{i}$ is a
system of holomorphic local coordinates on $\tilde{M}_{q_{\lambda}}$
for each $\lambda\in D$,
$\dot{\nu}=-(\tilde{X}_{i})_{\overline{z}}$
defines a Beltrami differential
on $\tilde{M}_{q_{0}}$ which
descends to a Beltrami differential $\dot{\mu}$
on $M_{q_{0}}$.
Notice that
for any holomorphic quadratic differential $\varphi$ on $M_{q_{0}}$,
the quotient
$(\pi_{q_{0}})^{*}(\varphi)/\omega_{q_{0}}$ is a holomorphic $1$-form on $\tilde{M}_{q_{0}}$
and
\begin{align*}
\int_{M_{q_{0}}}\dot{\mu}\varphi
&=
\frac{1}{2}\int_{\tilde{M}_{q_{0}}}\dot{\nu}\cdot (\pi_{q_{0}})^{*}(\varphi)
=
\frac{1}{2}\int_{\tilde{M}_{q_{0}}}f_{\overline{z}}\frac{(\pi_{q_{0}})^{*}(\varphi)}{\omega_{q_{0}}}
\\
&=
\frac{1}{2}\int_{\tilde{M}_{q_{0}}}
\overline{\partial}
\left(f(z)\frac{(\pi_{q_{0}})^{*}(\varphi)(z)}{\omega_{q_{0}}(z)}dz
\right) \\
&=
\frac{1}{2}\int_{\tilde{M}_{q_{0}}}
d
\left(f(z)\frac{(\pi_{q_{0}})^{*}(\varphi)(z)}{\omega_{q_{0}}(z)}dz
\right)
=0.
\end{align*}
Hence,
$\dot{\mu}$ is an infinitesimally trivial Beltrami differenital
on $M_{q_{0}}$
by Teichm\"uller lemma (cf. Lemma 7.6 of \cite{MR1215481}).
This means that
the cohomology class
$$
[\{X_{ij}\}_{i,j}]\in 
H^{1}(\mathcal{V}_{0},\Theta_{M_{q_{0}}})
\cong H^{1}(M_{q_{0}},\Theta_{M_{q_{0}}})
$$
is trivial from \eqref{eq:subsec:background-tangent-space4}
since $\mathcal{V}_{0}$ is a Leray covering
(see also Theorem 3.5 of \cite{MR815922}).
Therefore,
there is a $0$-cochain $\{Z_{i}\}_{i}
\in C^{0}(\mathcal{V}_{0,i},\Theta_{M_{q_{0}}})$
such that $Z_{i}-Z_{j}=X_{ij}$
for all $i,j$.

Let $\zeta_{i}=L_{Z_{i}}(q_{0})
\in \Gamma(\mathcal{V}_{0,i},\Omega^{\otimes 2}_{M_{q_{0}}})$
and $\tilde{\zeta}_{i}=L_{\tilde{Z}_{i}}(\omega_{q_{0}}^{2})$
be the lift of $\zeta_{i}$.
Then,
$$
\tilde{\zeta}_{i}-\tilde{\zeta}_{j}=L_{\tilde{X}_{ij}}(\omega_{q_{0}}^{2})
$$
and hence $\{\tilde{\phi}_{i}-\tilde{\zeta}_{i}\}_{i}$ defines a holomorphic
quadratic differential $\tilde{\varphi}$ on $\tilde{M}_{q_{0}}$
which descends to a holomorphic quadratic differential
$\varphi$ on $M_{q_{0}}$
associated to $\{\phi_{i}-\zeta_{i}\}_{i}$.
Since $\tilde{X}_{i}-\tilde{Z}_{i}=\tilde{X}_{j}-\tilde{Z}_{j}$ on $\mathcal{V}_{i,0}\cap \mathcal{V}_{j,0}$,
$\{\tilde{X}_{i}-\tilde{Z}_{i}\}_{i}$ defines a vector field $\tilde{W}$ on $\tilde{M}_{q_{0}}$
which satisfies
\begin{align*}
&((\omega_{q_{0}})'(\tilde{X}_{i})+
\omega_{q_{0}}\cdot (\tilde{X}_{i})_{z})dz
+
\omega_{q_{0}}\cdot (\tilde{X}_{i})_{\overline{z}}d\overline{z}
-
\frac{L_{\tilde{Z}_{i}}(\omega_{q_{0}}^{2})}{2\omega_{q_{0}}} \\
&=
\left((\omega_{q_{0}})'(\tilde{X}_{i}-\tilde{Z}_{i})+
\omega_{q_{0}}\cdot (\tilde{X}_{i}-\tilde{Z}_{i})_{z}
\right)dz
+
\omega_{q_{0}}\cdot (\tilde{X}_{i}-\tilde{Z}_{i})_{\overline{z}}d\overline{z} \\
&=d(\omega_{q_{0}}\cdot \tilde{W})
\end{align*}
since each $\tilde{Z}_{i}$ is a holomorphic section.
Consequently,
we deduce from the assumption that
\begin{equation}
\label{eq:lem:douady-hubbard-local-coordinate1}
0=
\int_{c}\Omega
=
\int_{c}\left(\Omega-\frac{\tilde{\zeta}_{i}-L_{\tilde{Z}_{i}}(\omega_{q_{0}}^{2})}{2\omega_{q_{0}}}
\right)
=
\int_{c}\frac{\tilde{\varphi}}{2\omega_{q_{0}}}
-
\int_{c}d(\omega_{q_{0}}\cdot W)
=
\int_{c}\frac{\tilde{\varphi}}{2\omega_{q_{0}}}
\end{equation}
for all $c\in H_{1}(\tilde{M}_{q_{0}})$,
and hence $\tilde{\varphi}=0$ and $\varphi=0$.
Therefore $\phi_{i}=\zeta_{i}$ and
$$
(\{\phi_{i}\}_{i},\{X_{ij}\}_{i,j})=
(\{L_{Z_{i}}(q_{0})\}_{i},\{Z_{i}-Z_{j}\}_{i,j}).
$$
Thus,
the cohomology class of the cochain $(\{\phi_{i}\}_{i},\{X_{ij}\}_{i,j})$
vanishes.

As a consequence,
the derivative of the map $\douadyhubbard$ on
the tangent space
$T_{q_{0}}\quaddiff{g,m} \cong \mathbb{H}^{1}(L^{\bullet})$
is injective.
Since the dimensions of $\quaddiff{g,m}$ and ${\rm Hom}(H_{1}(\tilde{M}_{q_{0}})^{-},\mathbb{C})$ are same,
the derivative of $\douadyhubbard$ is an isomorphism
from $T_{q_{0}}\quaddiff{g,m}$
onto ${\rm Hom}(H_{1}(\tilde{M}_{q_{0}})^{-},\mathbb{C})$. 
\end{proof}

%
%

\subsection{Homomorphisms on slices}
\label{subsec:Hubbard-Masur-Slice}
For $F\in \mathcal{MF}-\{0\}$.
we define a slice
$$
\slice(F)=\{q\in \quaddiff{g,m}\mid \vertical(q)=F\}.
$$
for $F$ in $\quaddiff{g,m}$.
Hubbard and Masur showed that the restriction of the projection
$$
\slice(F)\ni q\mapsto \pi (q)\in \teich{g,m}
$$
is a homeomorphism
(\cite[\S2]{MR523212}).
The following two lemmas are well-known
(cf. \cite[Lemma 4.3]{MR523212} and \cite[Proposition 1]{MR1283559}.
See also \cite{2012arXiv1203.0273D}).

\begin{lemma}
\label{lem:constant-representation}
Let $q_{0}\in \quaddiff{g,m}$ be a generic differential.
When we take $U$ sufficiently small,
$$
\slice(\vertical(q_{0}))\cap U=\{q\in \quaddiff{g,m}\mid {\rm Re}(\douadyhubbard(q))=
{\rm Re}(\douadyhubbard(q_{0}))\}
$$
holds.
\end{lemma}

\begin{lemma}
\label{lem:real-analytic-section}
Let $x_{0}\in \teich{g,m}$ and $F\in \mathcal{MF}$.
If $q_{F,x_{0}}$ is generic,
then
the map
$$
\teich{g,m}\ni x\mapsto q_{F,x}\in \quaddiff{g,m}
$$
is real analytic around $x_{0}$.
\end{lemma}
%

\section{Family of quadratic differentials with prescribed vertical foliation}
Fix $F\in \mathcal{MF}$.
Consider a family $\{q_{\lambda}\}_{\lambda\in D}$
of holomorphic quadratic differentials such that $\vertical(q_{\lambda})=F$.
Suppose that $q_{0}$ is generic
and $x_{\lambda}=\pi(q_{\lambda})\in \teich{g,m}$ varies holomorphically.
Then,
the family $\{q_{\lambda}\}_{\lambda\in D}$ is a real analytic family of quadratic
differentials
(cf. Lemma \ref{lem:real-analytic-section}).
Take a neighborhood $U$ of $q_{0}$ as the previous sections.

\subsection{Differential via the complex conjugate}
\label{subsec:Differentical-via-the-complex-conjugate}
%
We first notice a simple observation.
Let $\pi\colon E\to M$ be a holomorphic vector bundle.
Let $g\colon \mathbb{D}\to M$ be a holomorphic mapping
and $G\colon \mathbb{D}\to E$ be a $C^{1}$ mapping
such that $\pi\circ G=g$ on $\mathbb{D}$.
Then,
$$
\pi_{*}\left(G_{*}
\left(\frac{\partial}{\partial \overline{\lambda}}\right)
\right)
=g_{*}
\left(\frac{\partial}{\partial \overline{\lambda}}\right)
=0
$$
since $g$ is holomorphic.
Hence,
\begin{equation}
\label{eq:fiber-partial}
\left.G_{*}\left(\frac{\partial}{\partial \overline{\lambda}}
\right)
\right|_{\lambda}\in {\rm Ker}(\pi_{*})\cong E_{g(\lambda)}
\end{equation}
for all $\lambda\in \mathbb{D}$.
%
%
%
%
\begin{lemma}[First derivatives]
\label{lem:transversal}
Let $\{q_{\lambda}\}_{\lambda\in D}$ be 
the family of holomorphic quadratic differentials defined as above.
Let $\eta_{\lambda}$ be a holomorphic quadratic differential
on $X_{q_{\lambda}}$
such that the infinitesimal Bertrami differential
$\dot{\mu}$
associated to the tangent vector 
for the $\lambda$-derivative of the mapping
$$
D\ni \lambda\mapsto \pi(q_{\lambda})\in \teich{g,m}
$$
at $\lambda$ satisfies
$$
\int_{M_{q_{\lambda}}}\dot{\mu}\psi
=
\int_{M_{q_{\lambda}}}\frac{\overline{\eta_{\lambda}}}{|q_{\lambda}|}
\psi
$$
for all holomorphic quadratic differential $\psi$ on $M_{q_{\lambda}}$.
Then,
$$
\frac{\partial\,\chi_{q_{\lambda}}(c)}{\partial \lambda}
=\int_{c}
\overline{
\left(
\frac{(\pi_{q_{\lambda}})^{*}(\eta_{\lambda})}{\omega_{q_{\lambda}}}
\right)
},
\quad
\frac{\partial\,\chi_{q_{\lambda}}(c)}{\partial \overline{\lambda}}
=-\int_{c}\frac{(\pi_{q_{\lambda}})^{*}(\eta_{\lambda})}{\omega_{q_{\lambda}}}
$$
for any $c\in H_{1}(\tilde{M}_{q_{0}})$ and $\lambda\in D$.
\end{lemma}

\begin{proof}
The existence and the uniqueness of $\eta_{\lambda}$ follows from
an argumant by David Dumas \cite[Theorem 5.3]{2012arXiv1203.0273D}
since each $q_{\lambda}$ is generic.
Indeed,
the pairing
$$
(\eta,\psi)\mapsto \int_{M_{q_{\lambda}}}\frac{\overline{\eta}\psi}{|q_{\lambda}|}
$$
is a positive definite hermitian inner product
on the space of holomorphic quadratic differentials on $M_{q_{\lambda}}$.
Notice that
$$
\frac{\partial\,\chi_{q_{\lambda}}(c)}{\partial \overline{\lambda}}
=
0=\int_{c}\frac{(\pi_{q_{\lambda}})^{*}(\eta_{\lambda})}{\omega_{q_{\lambda}}}
$$
for any $c\in H_{1}(\tilde{M}_{q_{0}})^{+}$ and $\lambda\in D$.
It suffices to show the case where $\lambda=0$
and $c\in H_{1}(\tilde{M}_{q_{0}})^{-}$.
As we discussed around \eqref{eq:fiber-partial},
from Lemma \ref{lem:Douady-Hubbard}
and the proof of Lemma \ref{lem:douady-hubbard-local-coordinate},
there is a holomorphic quadratic differential $\varphi$ on $X_{q_{0}}$
such that
$$
\left.
\frac{\partial\,\chi_{q_{\lambda}}(c)}{\partial \overline{\lambda}}
\right|_{\lambda=0}
=\int_{c}\frac{(\pi_{q_{0}})^{*}(\varphi)}{\omega_{q_{0}}}
$$
for any $c\in H_{1}(\tilde{M}_{q_{0}})^{-}$
(see \eqref{eq:lem:douady-hubbard-local-coordinate1}).
Since the vertical foliation $|{\rm Re}(\omega_{q_{\lambda}})|$
of $\omega_{q_{\lambda}}$ is unchanged
on $D$,
for all $c\in H_{1}(\tilde{M}_{q_{0}})^{-}$,
the real part of $\chi_{q_{\lambda}}(c)$ is a constant function on $D$
by Lemma \ref{lem:constant-representation}.
Therefore,
$$
\frac{\partial\,\chi_{q_{\lambda}}(c)}{\partial \overline{\lambda}}
=-
\overline{\frac{\partial\,\chi_{q_{\lambda}}(c)}{\partial \lambda}}
$$
on $D$.
Let $(\{\phi_{i}\}_{i},\{X_{ij}\}_{i,j})$
be the cocycle representing the tangent vector
associated to the $\lambda$-derivative
of the mapping $D\ni \lambda\mapsto q_{\lambda}$
at $\lambda=0$.
Let $\{X_{i}\}_{i}$ be a $0$-cochain of the sheaf
of $C^{\infty}$-vector fields with $X_{i}-X_{j}=X_{ij}$.
Then,
$$
\Omega=\left(\frac{\tilde{\phi}_{i}}{2\omega_{q_{0}}}
-\omega_{q_{0}}'\tilde{X}_{i}-\omega_{q_{0}}(\tilde{X}_{i})_{z}
\right)dz
-
(\omega_{q_{0}}(\tilde{X}_{i})_{\overline{z}})
d\overline{z}
$$
on $U_{i}$ defines a closed $1$-form on $\tilde{M}_{q_{0}}$
satisfies
$$
\frac{\partial\,\chi_{q_{\lambda}}(c)}{\partial \lambda}
=\int_{c}\Omega
$$
for all $c\in H_{1}(\tilde{M}_{q_{0}})$.
Therefore,
there is a $C^{\infty}$-function $f$ on $\tilde{M}_{q_{0}}$
such that
$$
\Omega=-\overline{\frac{(\pi_{q_{0}})^{*}(\varphi)(z)}{\omega_{q_{0}}(z)}}d\overline{z}+df,
$$
that is,
\begin{align}
\frac{\tilde{\phi}_{i}}{2\omega_{q_{0}}}
-\omega_{q_{0}}'\tilde{X}_{i}-\omega_{q_{0}}(\tilde{X}_{i})_{z}
&=f_{z}
\label{eq:lem:transversal1}
\\
-\omega_{q_{0}}(\tilde{X}_{i})_{\overline{z}}
&=
-\overline{\left(\frac{(\pi_{q_{0}})^{*}(\varphi)(z)}{\omega_{q_{0}}(z)}\right)}
+f_{\overline{z}}
\label{eq:lem:transversal2}
\end{align}
on $\mathcal{V}_{i,0}$.
From
\eqref{eq:lem:transversal2},
the lift $(\pi_{q_{0}})^{*}(\dot{\mu})$ of
$\dot{\mu}=-(X_{i})_{\overline{z}}$
satisfies
\begin{equation}
\label{eq:lem:transversal3}
(\pi_{q_{0}})^{*}(\dot{\mu})
=
-\frac{\overline{(\pi_{q_{0}})^{*}(\varphi)}}{|\omega_{q_{0}}|^{2}}
+\frac{f_{\overline{z}}}{\omega_{q_{0}}}
\end{equation}
on $\tilde{M}_{q_{0}}$.
Let $\psi$ be a holomorphic quadratic differenital on $M_{q_{0}}$.
From \eqref{eq:lem:transversal3}
\begin{align*}
\int_{\tilde{M}_{q_{0}}}
(\pi_{q_{0}})^{*}(\dot{\mu})
(\pi_{q_{0}})^{*}(\psi)
&=
-\int_{\tilde{M}_{q_{0}}}
\frac{\overline{(\pi_{q_{0}})^{*}(\varphi)}}{|\omega_{q_{0}}|^{2}}
(\pi_{q_{0}})^{*}(\psi)
+
\int_{\tilde{M}_{q_{0}}}
\frac{f_{\overline{z}}}{\omega_{q_{0}}}
(\pi_{q_{0}})^{*}(\psi) \\
&=
-\int_{\tilde{M}_{q_{0}}}
\frac{\overline{(\pi_{q_{0}})^{*}(\varphi)}}{|\omega_{q_{0}}|^{2}}
(\pi_{q_{0}})^{*}(\psi)
+
\int_{\tilde{M}_{q_{0}}}
d\left(
f\frac{(\pi_{q_{0}})^{*}(\psi)}{\omega_{q_{0}}}
\right) \\
&=
-\int_{\tilde{M}_{q_{0}}}
\frac{\overline{(\pi_{q_{0}})^{*}(\varphi)}}{|\omega_{q_{0}}|^{2}}
(\pi_{q_{0}})^{*}(\psi)
\end{align*}
By descending to $M_{q_{0}}$,
we have
$$
\int_{M_{q_{0}}}\frac{\overline{\eta_{0}}}{|q_{0}|}
\psi
=
\int_{M_{q_{0}}}\dot{\mu}\psi
=
-\int_{M_{q_{0}}}\frac{\overline{\varphi}\psi}{|q_{0}|}
$$
for all $\psi$.
Therefore,
$\varphi=-\eta_{0}$ from 
\cite[Theorem 5.3]{2012arXiv1203.0273D}
again.
\end{proof}

\subsection{Laplacian of homomorphisms}
This section is devoted to calculating
the Laplacian of Douady-Hubbard map
\eqref{eq:Douady-Hubbard}.
The author must confess that
for the proof of the plurisubharmonicity of extremal length funcitons,
we need the equation \eqref{eq:Omega-0-Laplacian}
rather than the formula \eqref{lem:transversal2-1} in the following lemma.
However,
we give the following lemma for its own interests.

\begin{lemma}[Laplacian]
\label{lem:transversal2}
Under the assumption in Lemma \ref{lem:transversal2},
we have
\begin{equation}
\label{lem:transversal2-1}
\left.
\frac{\partial^{2}(\chi_{q_{\lambda}}(c))}{\partial \lambda\partial\overline{\lambda}}
\right|_{\lambda=0}
=
-4i\,{\rm Im}
\int_{\tilde{M}_{q_{0}}}
\frac{|(\pi_{q_{0}})^{*}(\eta_{0})|^{2}}{|\omega_{q_{0}}|^{2}}
\frac{\sigma_{c}}{\omega_{q_{0}}}
\end{equation}
for $c\in H_{1}(\tilde{M}_{q_{0}})$,
where $\sigma_{c}=\sigma_{c}(z)dz$ is the holomorphic part of the
reproducing harmonic differential
associated to $c$.
\end{lemma}

\begin{proof}
When the homology class $c$ is in the even homology group,
$\chi_{q_{\lambda}}(c)=0$.
The integrand of the right-hand side of \eqref{lem:transversal2-1}
satisfies
$$
r_{0}^{*}\left(\frac{|(\pi_{q_{0}})^{*}(\eta_{0})|^{2}}{|\omega_{q_{0}}|^{2}}
\frac{\sigma_{c}}{\omega_{q_{0}}}\right)
=
\frac{|(\pi_{q_{0}})^{*}(\eta_{0})|^{2}}{|-\omega_{q_{0}}|^{2}}
\frac{\sigma_{c}}{-\omega_{q_{0}}}
=
-
\frac{|(\pi_{q_{0}})^{*}(\eta_{0})|^{2}}{|\omega_{q_{0}}|^{2}}
\frac{\sigma_{c}}{\omega_{q_{0}}}
$$
and hence,
the integral over $\tilde{M}_{q_{0}}$ vanishes.
This means that \eqref{lem:transversal2-1} holds for homology classes
in the even homology group.

By applying the same argument as that of Lemma \ref{lem:Douady-Hubbard},
one obtain
$$
\left.
\frac{\partial^{2}\,\chi_{q_{\lambda}}(c)}{\partial\lambda\partial \overline{\lambda}}
\right|_{\lambda=0}
=\int_{c}\Omega_{0}
$$
for any $c\in H_{1}(\tilde{M}_{q_{0}})$,
where $\Omega_{0}$ is a closed differential of the form
\begin{align*}
\left(
\dot{\eta}_{0}
+\left(\frac{(\pi_{q_{0}})^{*}(\eta_{0})}{\omega_{q_{0}}}\right)'(\tilde{X}_{i})
+
\left(\frac{(\pi_{q_{0}})^{*}(\eta_{0})}{\omega_{q_{0}}}\right)(\tilde{X}_{i})_{z}
\right)dz 
+
\left(
\frac{(\pi_{q_{0}})^{*}(\eta_{0})}{\omega_{q_{0}}}\right)(\tilde{X}_{i})_{\overline{z}}
d\overline{z}
\end{align*}
and
$$
\dot{\eta}_{0}
=
-\left.
\frac{\partial}{\partial\lambda}\left(\frac{(\pi_{q_{\lambda}})^{*}(\eta_{\lambda})}{\omega_{q_{\lambda}}}\right)
\right|_{\lambda=0}.
$$
Since the real part of
the function $D\ni \lambda\mapsto \chi_{q_{\lambda}}(c)$
is constant
and the Laplacian is the real operator,
$$
\int_{c}\Omega_{0}=
\left.
\frac{\partial^{2}\,\chi_{q_{\lambda}}(c)}{\partial\lambda\partial \overline{\lambda}}
\right|_{\lambda=0}
=
-\overline{\left.
\frac{\partial^{2}\,\chi_{q_{\lambda}}(c)}{\partial\lambda\partial \overline{\lambda}}
\right|_{\lambda=0}
}
=
-\overline{\left(\int_{c}\Omega_{0}\right)}
=-\int_{c}\overline{\Omega_{0}}
$$
for all $c\in H_{1}(\tilde{M}_{q_{0}})$.
Hence,
there is a $C^{\infty}$-function $f$ on $\tilde{M}_{q_{0}}$
such that
$$
\Omega_{0}=-\overline{\Omega_{0}}+df.
$$
Namely,
we have
\begin{align*}
\dot{\eta}_{0}
+\left(\frac{(\pi_{q_{0}})^{*}(\eta_{0})}{\omega_{q_{0}}}\right)'(X_{i})
+
\left(\frac{(\pi_{q_{0}})^{*}(\eta_{0})}{\omega_{q_{0}}}\right)(\tilde{X}_{i})_{z}
&=
-\overline{
\left(
\frac{(\pi_{q_{0}})^{*}(\eta_{0})}{\omega_{q_{0}}}\right)(\tilde{X}_{i})_{\overline{z}}}
+f_{z} \\
&=
\overline{
\left(
\frac{(\pi_{q_{0}})^{*}(\eta_{0})}{\omega_{q_{0}}}\right)
}
(\pi_{q_{0}})^{*}(\dot{\mu})
+f_{z},
\end{align*}
where $\dot{\mu}=-(X_{i})_{\overline{z}}$ is the infinitesimal Beltrami
differential representing the tangent vector 
of the mapping $D\ni \mapsto \pi(q_{\lambda})$
at $\lambda=0$
(cf. \eqref{eq:infinitesimal1}).
Applying the above equation,
we also obtain
\begin{align}
\Omega_{0}
&=
\left(
\overline{
\left(
\frac{(\pi_{q_{0}})^{*}(\eta_{0})}{\omega_{q_{0}}}\right)
(\pi_{q_{0}})^{*}(\dot{\mu})
}
+f_{z}
\right)dz
-
\left(
\frac{(\pi_{q_{0}})^{*}(\eta_{0})}{\omega_{q_{0}}}\right)
(\pi_{q_{0}})^{*}(\dot{\mu})
d\overline{z}
\nonumber
\\
&=
\overline{
\left(
\frac{(\pi_{q_{0}})^{*}(\eta_{0})}{\omega_{q_{0}}}\right)
(\pi_{q_{0}})^{*}(\dot{\mu})
}
dz
-
\left(
\frac{(\pi_{q_{0}})^{*}(\eta_{0})}{\omega_{q_{0}}}\right)
(\pi_{q_{0}})^{*}(\dot{\mu})
d\overline{z}
\label{eq:Omega-0-Laplacian}
\\
&\qquad\qquad+\frac{1}{2}df+\frac{i}{2}\,{}^{*}df.
\nonumber
\end{align}
Let $\sigma(c)=\sigma_{c}(z)dz+\overline{\sigma_{c}(z)}d\overline{z}$
be the reproducing harmonic differential on $\tilde{M}_{q_{0}}$
associated to $c\in H_{1}(\tilde{M}_{q_{0}})$,
where $\sigma_{c}(z)dz$ is a holomorphic $1$-form on $\tilde{M}_{q_{0}}$
(we call this holomorphic $1$-form the \emph{holomorphic part} of $\sigma(c)$).
Namely,
the differential $\sigma(c)$ is a unique harmonic differential on $\tilde{M}_{q_{0}}$
which satisfies
$$
\int_{c}\omega=\int_{\tilde{M}_{q_{0}}}\omega\wedge {}^{*}\overline{\sigma(c)}
$$
for all closed $1$-form $\omega$ on $\tilde{M}_{q_{0}}$.
Then,
by applying
the orthogonal decomposition theorem
for the space of $L^{2}$-closed forms,
we have
\begin{align*}
\left.
\frac{\partial^{2}\,\chi_{q_{\lambda}}(c)}{\partial\lambda\partial \overline{\lambda}}
\right|_{\lambda=0}
&=
\int_{c}\Omega_{0}
=\int_{\tilde{M}_{q_{0}}}\Omega_{0}\wedge {}^{*}\overline{\sigma(c)}\\
&=\int_{\tilde{M}_{q_{0}}}
\left(
\overline{\frac{(\pi_{q_{0}})^{*}(\eta_{0})}{\omega_{q_{0}}}
(\pi_{q_{0}})^{*}(\dot{\mu})}\,
dz-
\frac{(\pi_{q_{0}})^{*}(\eta_{0})}{\omega_{q_{0}}}
(\pi_{q_{0}})^{*}(\dot{\mu})\,
d\overline{z}
\right) \\
&\quad\quad\qquad\qquad
\wedge
{}^{*}\overline{(\sigma_{c}(z)dz+\overline{\sigma_{c}(z)}d\overline{z})}
\\
&=
\int_{\tilde{M}_{q_{0}}}
\left(
\overline{\frac{(\pi_{q_{0}})^{*}(\eta_{0})}{\omega_{q_{0}}}
(\pi_{q_{0}})^{*}(\dot{\mu})}\,
dz-
\frac{(\pi_{q_{0}})^{*}(\eta_{0})}{\omega_{q_{0}}}
(\pi_{q_{0}})^{*}(\dot{\mu})\,
d\overline{z}
\right) \\
&\quad\quad\qquad\qquad
\wedge
(-i\overline{\sigma_{c}}d\overline{z}+i\sigma_{c}dz)
\\
&=
-4i\,{\rm Im}
\int_{\tilde{M}_{q_{0}}}
\frac{(\pi_{q_{0}})^{*}(\eta_{0})}{\omega_{q_{0}}}
(\pi_{q_{0}})^{*}(\dot{\mu})\sigma_{c}.
\end{align*}
Now,
we suppose that the homology class $c$ is in $H_{1}(\tilde{M}_{q_{0}})^{-}$.
Since $(r_{0})_{*}(c)=-c$,
$r_{0}^{*}\sigma(c)=-\sigma(c)$.
This means that $\sigma(c)$ has zeroes at each branch points of
$\pi_{q_{0}}\colon
\tilde{M}_{q_{0}}\to M_{q_{0}}$
and the holomorphic quadratic differential
$$
\frac{(\pi_{q_{0}})^{*}(\eta_{0})}{\omega_{q_{0}}}\sigma_{c}=
\frac{(\pi_{q_{0}})^{*}(\eta_{0})(z)}{\omega_{q_{0}}(z)}\sigma_{c}(z)dz^{2}
$$
descends to a holomorphic quadratic differential on $M_{q_{0}}$.
Therefore,
we obtain
from the definition of $\eta_{0}$
that
\begin{align}
\left.
\frac{\partial^{2}\,\chi_{q_{\lambda}}(c)}{\partial\lambda\partial \overline{\lambda}}
\right|_{\lambda=0}
&=
-4i\,{\rm Im}
\int_{\tilde{M}_{q_{0}}}
\frac{(\pi_{q_{0}})^{*}(\eta_{0})}{\omega_{q_{0}}}
\frac{\overline{(\pi_{q_{0}})^{*}(\eta_{0})}}{|\omega_{q_{0}}|^{2}}\sigma_{c}
\nonumber
\\
&=
-4i\,{\rm Im}
\int_{\tilde{M}_{q_{0}}}
\frac{|(\pi_{q_{0}})^{*}(\eta_{0})|^{2}}{|\omega_{q_{0}}|^{2}}
\frac{\sigma_{c}}{\omega_{q_{0}}}
\label{eq:laplacian-1}
\end{align}
(cf. Lemma \ref{lem:transversal}).
This is what we wanted.
\end{proof}


\section{Levi forms and Plurisubharmonicity of subspecies}
\subsection{Levi forms of extremal length functions}
For a real-valured $C^{2}$-function  $U$  on a complex manifold $N$,
the \emph{Levi form} of $U$ is an hermitian inner product on
the holomorphic tangent bundle $TN$ of $N$ defined as
$$
\levi{U}[v,\overline{v}]
=\sum_{i,j=1}^{n}
\frac{\partial^{2}U}{\partial z_{i}\partial \overline{z}_{j}}
v_{i}\overline{v_{j}}
$$
for $v=\sum_{i=1}^{n}v_{i}(\partial/\partial z_{i})\in TN$.
Let $p\in N$ and $v\in T_{p}N$.
Let $F\colon \mathbb{D}\to N$ be a holomorphic mapping
with $f(0)=p$ and $f_{*}(\partial/\partial\lambda\mid_{\lambda=0})=v$,
then
$$
\levi{U}[v,\overline{v}]=
\left.
\frac{\partial^{2}(U\circ f)}{\partial\lambda\partial\overline{\lambda}}
\right|_{\lambda=0}.
$$
\begin{theorem}[Levi forms]
\label{thm:leviform}
Let $x_{0}=(M_{0},f_{0})\in \teich{g,m}$ and $v=[\dot{\mu}]\in T_{x_{0}}\teich{g,m}$.
Let $F\in \mathcal{MF}$.
Suppose that the Hubbard-Masur differential $q_{0}=q_{F,x_{0}}$ is generic.
Then,
the extremal length function is real analytic around $x_{0}$,
and
the Levi form of the extremal length function
$\teich{g,m}\ni x\mapsto \ext_{x}(F)$
at $x_{0}$ satisfies
$$
\levi{\ext_{\cdot }(F)}[v,\overline{v}]
=
2\int_{M_{0}}\frac{|\eta_{v}|^{2}}{|q_{0}|},
$$
where $\eta_{v}$ is a unique
holomorphic quadratic differential on $M_{0}=M_{q_{0}}$
satisfying that
$$
\langle v,\psi\rangle=
\int_{M_{0}}\dot{\mu}\psi=
\int_{M_{0}}\frac{\overline{\eta_{v}}}{|q_{0}|}\psi
$$
for all $\psi\in \quaddiff{x_{0}}$.
\end{theorem}

\begin{remark}
In \S\ref{sec:geometric-interpretation},
we will discuss a geometric interpretation
of the anti-complex linear map
$$
T_{x_{0}}\teich{g,m}\ni
v\mapsto \eta_{v}
\in \quaddiff{x_{0}}.
$$
\end{remark}

\begin{proof}[Proof of Theorem \ref{thm:leviform}]
We continue to use the symbols defined in the previous sections.
Let $\tilde{g}=4g-2+m$ $(\ge 2)$ be the genus of $\tilde{M}_{q_{0}}$.
Let $\{\alpha_{k},\beta_{k}\}_{k=1}^{\tilde{g}}$ be
a symplectic generator of $H_{1}(\tilde{M}_{q_{0}})$.
Let $x_{\lambda}=\pi(q_{\lambda})$ for instance.
By the bilinear relation,
we have
\begin{align*}
\ext_{x_{\lambda}}(F)
&=\frac{1}{2}\|\omega_{q_{\lambda}}\|^{2}
=\frac{1}{4}\int_{\tilde{M}_{q_{\lambda}}}\omega_{q_{\lambda}}\wedge
{}^{*}\overline{\omega_{q_{\lambda}}}
\\
&=\frac{i}{4}\sum_{k=1}^{\tilde{g}}
\left(
\chi_{q_{\lambda}}(\alpha_{k})\overline{\chi_{q_{\lambda}}(\beta_{k})}
-\chi_{q_{\lambda}}(\beta_{k})\overline{\chi_{q_{\lambda}}(\alpha_{k})}
\right)
\end{align*}
for $\lambda\in D$
(cf. \cite[Corollary in \S III.2.3]{MR1139765}).
From Lemmas \ref{lem:douady-hubbard-local-coordinate}
and \ref{lem:constant-representation},
$\chi_{q_{F,x}}(\alpha_{k})$ and $\chi_{q_{F,x}}(\beta_{k})$
are real analytic  around $x_{0}$ for each $k$,
and hence
so is the extremal length function.
From Lemma \ref{lem:transversal2},
we have
\begin{align*}
&
\left.
\frac{\partial^{2}}{\partial\lambda\partial\overline{\lambda}}
\left(\chi_{q_{\lambda}}(c_{1})\overline{\chi_{q_{\lambda}}(c_{2})}
\right)
\right|_{\lambda=0}\\
&=
\left.
\frac{\partial^{2}\chi_{q_{\lambda}}(c_{1})}{\partial\lambda\partial\overline{\lambda}}
\right|_{\lambda=0}
\overline{\chi_{q_{0}}(c_{2})}
+\chi_{q_{0}}(c_{1})
\left.
\frac{\partial^{2}\overline{\chi_{q_{\lambda}}(c_{2})}}{\partial\lambda\partial\overline{\lambda}}
\right|_{\lambda=0} \\
&\quad\quad
+\left.
\frac{\partial \chi_{q_{\lambda}}(c_{1})}{\partial\lambda}
\right|_{\lambda=0}
\left.
\frac{\partial \overline{\chi_{q_{\lambda}}(c_{2})}}{\partial\overline{\lambda}}
\right|_{\lambda=0}
+
\left.
\frac{\partial \chi_{q_{\lambda}}(c_{1})}{\partial\overline{\lambda}}
\right|_{\lambda=0}
\left.
\frac{\partial \overline{\chi_{q_{\lambda}}(c_{2})}}{\partial \lambda}
\right|_{\lambda=0}
\\
&=
\left.
\frac{\partial^{2}\chi_{q_{\lambda}}(c_{1})}{\partial\lambda\partial\overline{\lambda}}
\right|_{\lambda=0}
\overline{\chi_{q_{0}}(c_{2})}
+\chi_{q_{0}}(c_{1})
\left.
\frac{\partial^{2}\overline{\chi_{q_{\lambda}}(c_{2})}}{\partial\lambda\partial\overline{\lambda}}
\right|_{\lambda=0} \\
&\quad\quad
+\left.
\frac{\partial \chi_{q_{\lambda}}(c_{1})}{\partial\lambda}
\right|_{\lambda=0}
 \overline{
 \left.
\frac{\partial\chi_{q_{\lambda}(c_{2})}}{\partial\lambda}
\right|_{\lambda=0}
}
+
 \overline{
 \left.
\frac{\partial \chi_{q_{\lambda}}(c_{1})}{\partial\lambda}
\right|_{\lambda=0}
}
\left.
\frac{\partial \chi_{q_{\lambda}}(c_{2})}{\partial \lambda}
\right|_{\lambda=0}
\end{align*}
for $c_{1},c_{2}\in H_{1}(\tilde{M}_{q_{0}})$,
since the real part of $\chi_{q_{\lambda}}(c)$ is constant
for all $c\in H_{1}(\tilde{M}_{q_{0}})$
by Lemma \ref{lem:constant-representation}.
Notice that the sum of the last two terms of above calculation is a real number.
Therefore,
we obtain
\begin{align*}
&
\left.
\frac{\partial^{2}}{\partial\lambda\partial\overline{\lambda}}\left(
\chi_{q_{\lambda}}(\alpha_{k})\overline{\chi_{q_{\lambda}}(\beta_{k})}
-\chi_{q_{\lambda}}(\beta_{k})\overline{\chi_{q_{\lambda}}(\alpha_{k})}
\right)
\right|_{\lambda=0}\\
&=
\int_{\alpha_{k}}\Omega_{0}\overline{\int_{\beta_{k}}\omega_{q_{0}}}
+\overline{\int_{\beta_{k}}\Omega_{0}}
\int_{\alpha_{k}}\omega_{q_{0}}
-
\int_{\beta_{k}}\Omega_{0}\overline{\int_{\alpha_{k}}\omega_{q_{0}}}
-\int_{\alpha_{k}}\overline{\Omega_{0}}
\int_{\beta_{k}}\omega_{q_{0}} \\
&=
\int_{\alpha_{k}}\Omega_{0}\int_{\beta_{k}}\overline{\omega_{q_{0}}}
-
\int_{\beta_{k}}\Omega_{0}\int_{\alpha_{k}}\overline{\omega_{q_{0}}}
+\int_{\beta_{k}}\overline{\Omega_{0}}
\int_{\alpha_{k}}\omega_{q_{0}}
-\int_{\alpha_{k}}\overline{\Omega_{0}}
\int_{\beta_{k}}\omega_{q_{0}}
\end{align*}
where $\Omega_{0}$ is a closed form defined as \eqref{eq:Omega-0-Laplacian}.
Since $\omega_{q_{0}}$ is a holomorphic differential,
${}^{*}\omega_{q_{0}}$
and ${}^{*}\overline{\omega_{q_{0}}}$
are harmonic differentials and they are
perpendicular to the exact and co-exact differentials
in the $L^{2}$-inner product
of differential forms
(cf. \cite[\S II.3]{MR1139765}).
From \eqref{eq:Omega-0-Laplacian},
we obtain
\begin{align*}
\levi{\ext_{\cdot }(F)}[v,\overline{v}]
&=
\left.
\frac{\partial^{2}}{\partial\lambda\partial\overline{\lambda}}
\ext_{x_{\lambda}}(F)
\right|_{\lambda=0}\\
&=
\frac{i}{4}
\sum_{k=1}^{\tilde{g}}
\left(
\int_{\alpha_{k}}\Omega_{0}\int_{\beta_{k}}\overline{\omega_{q_{0}}}
-
\int_{\beta_{k}}\Omega_{0}\int_{\alpha_{k}}\overline{\omega_{q_{0}}}
\right) \\
&\quad
+
\frac{i}{4}
\sum_{k=1}^{\tilde{g}}
\left(
\int_{\beta_{k}}\overline{\Omega_{0}}
\int_{\alpha_{k}}\omega_{q_{0}}
-\int_{\alpha_{k}}\overline{\Omega_{0}}
\int_{\beta_{k}}\omega_{q_{0}}
\right) \\
&=\frac{i}{4}
\int_{\tilde{M}_{q_{0}}}\Omega_{0}\wedge \overline{\omega_{q_{0}}}
-
\frac{i}{4}
\int_{\tilde{M}_{q_{0}}}\overline{\Omega_{0}}\wedge \omega_{q_{0}}
\\
&={\rm Re}
\int_{\tilde{M}_{q_{0}}}
\frac{(\pi_{q_{0}})^{*}(\eta_{v})}{\omega_{q_{0}}}
(\pi_{q_{0}})^{*}(\dot{\mu})\omega_{q_{0}}
\\
&=
2{\rm Re}
\int_{M_{q_{0}}}
\eta_{v}\dot{\mu}
=
2\int_{M_{0}}
\frac{|\eta_{v}|^{2}}{|q_{0}|}
\end{align*}
from the definition of the differential $\eta_{v}$
(cf. \cite[Proposition III.2.3]{MR1139765}).
This implies what we wanted.
\end{proof}

From Theorem \ref{thm:leviform},
we deduce the following inequality.

\begin{corollary}[Gradients and Levi forms]
\label{coro:strong-positivity}
Let $x_{0}=(M_{0},f_{0})\in \teich{g,m}$ and $v=[\dot{\mu}]\in T_{x_{0}}\teich{g,m}$.
Let $F\in \mathcal{MF}$.
Suppose that $q_{0}=q_{F,x_{0}}$ is generic.
Then,
$$
2\left|
(\ext_{\cdot}(F))_{*}[v]
\right|^{2}
\le 
\ext_{x_{0}}(F)\cdot \levi{\ext_{\cdot}(F)}[v,\overline{v}]
$$
holds.
\end{corollary}

\begin{proof}
It follows from Gardiner's formula 
that
\begin{equation}
\label{eq:Gardiner-formula1}
(\ext_{\cdot}(F))_{*}[v]
=
-\int_{M_{0}}\dot{\mu}\,q_{0}
=-\int_{M_{0}}\frac{\overline{\eta_{v}}}{|q_{0}|}q_{0}.
\end{equation}
See \cite[\S11.8]{MR903027}
(See also \S\ref{subsec:Appendix-Gardiner-formula}).
Notice in comparing the Gardiner's original formula
with \eqref{eq:Gardiner-formula1}
that our differential $q_{0}=q_{F,x_{0}}$ has $F$ as the vertical foliation,
while Gardiner considers
the quadratic differential with horizontal foliation $F$.
Hence,
we have to put the minus sign in \eqref{eq:Gardiner-formula1}.
From
Cauchy-Schwarz inequality,
one can see from Theorem \ref{thm:leviform}
that
\begin{align*}
&
\ext_{x_{0}}(F)\levi{\ext_{\cdot}(F)}[v,\overline{v}]
-2\left|
(\ext_{\cdot}(F))_{*}[v]
\right|^{2}
\\
&=
\ext_{x_{0}}(F)\cdot 2
\int_{M_{0}}
\frac{|\eta_{v}|^{2}}{|q_{0}|}
-2\left|
-\int_{M_{0}}\frac{\overline{\eta_{v}}}{|q_{0}|}q_{0}
\right|^{2} \\
&=
2\left(
\int_{M_{0}}
|q_{0}|
\cdot 
\int_{M_{0}}
\frac{|\eta_{v}|^{2}}{|q_{0}|}
-\left|
\int_{M_{0}}\frac{\overline{\eta_{v}}}{|q_{0}|^{1/2}}\frac{q_{0}}{|q_{0}|^{1/2}}
\right|^{2}
\right)
\ge 0,
\end{align*}
which implies what we wanted.
\end{proof}

\subsection{Plurisubharmonicity for subspecies}
\label{subsec:log-plurisubharmonicity}
Let $N$ be a complex manifold.
A function $U$ on $N$ is said to be \emph{plurisubharmonic} if 
$U$ is upper semi-continuous,
and $U\circ f$ is subharmonic for all holomorphic mapping
$f\colon \mathbb{D}\to N$.
A $C^{2}$-function on $N$ is plurisubharmonic
if the Levi form is positive semi-definite.
We call a $C^{2}$-function \emph{strictly plurisubharmonic}
if the Levi form is positive-definite.
When a function $U$ is plurisubharmonic,
so is $g\circ U$
for any increasing convex function $g$ on $\mathbb{R}$
such that the limit $\lim_{t\to -\infty}g(t)$ exists.

In the proof of the following theorem,
we remind that if $U(z)$ is a real-valued positive $C^{2}$-function on a domain
on $\mathbb{C}$,
$\log U$ and $-1/U$
satisfy
\begin{align}
(\log U)_{\lambda\overline{\lambda}}
&=\frac{U\cdot U_{\lambda\overline{\lambda}}-|U_{\lambda}|^{2}}{U^{2}}
\label{eq:log-laplacian}
\\
\left(-\frac{1}{U}\right)_{\lambda\overline{\lambda}}
&=\frac{U\cdot U_{\lambda\overline{\lambda}}-2|U_{\lambda}|^{2}}{U^{3}}.
\label{eq:reciprocal-laplacian}
\end{align}
Especially,
for a positive $C^{2}$-function $U$ on a complex manifold,
if $-1/U$ is plurisubharminic,
$\log U$ and $U$ are also plurisubharmonic.
From \eqref{eq:log-laplacian} and \eqref{eq:reciprocal-laplacian},
Theorem \ref{thm:pluri-subharmonicity} follows from
the following theorem.

\begin{theorem}[Plurisuperharmonicity of reciprocals]
\label{thm:plurisuperharmonicitiy-of-reciprocals}
Let $F_{1},F_{2},\cdots,F_{n}$ be
measured foliations.
Let $a_{k},c\ge0$ $(k=1,\cdots,n)$ with $\sum_{k=1}^{n}a_{k}>0$.
If $n\ge 2$,
$a_{k},c>0$
$(k=1,\cdots,n)$
and
 some two of them are transverse
 in the sense of
\eqref{eq:transverse-mf},
the function
$$
\rho(x)=-\frac{1}{c+\sum_{k=1}^{n}a_{k}\ext_{x}(F_{k})}
$$
is a negative,
plurisubharmonic exhaustion function
of $\teich{g,m}$
with lower bound.
When $n\ge 1$ or $c\ge 0$,
$\rho$ is a negative, 
plurisubharmonic function
on $\teich{g,m}$
with lower bound.
\end{theorem}

\begin{proof}
We only show the case $n=2$.
The other case can be treated in the same way.
Furthermore,
we may also assume that each $F_{k}$ are essentially complete,
and $F_{1}$ and $F_{2}$ are transverse
in the sense of
\eqref{eq:transverse-mf}.
The general case follows from a standard approximating procedure
(cf. \cite[Theorem 2.6.1]{MR1150978},
\cite[Theorem 2 in Chapter II]{MR0243112} or \cite[3.3 in Chapter II]{MR0344979}).
Under the assumption,
each extremal length function
$\teich{g,m}\ni x\mapsto \ext_{x}(F_{k})$ is 
real analytic.

Let $x=(M,f)\in \teich{g,m}$ and $v\in T_{x_{0}}\teich{g,m}$.
Let $h\colon \mathbb{D}\to \teich{g,m}$
be a holomorphic mapping with $h(0)=x_{0}$
and 
$h_{*}((\partial/\partial\lambda)\mid_{\lambda=0})=v$.
%
Set $E^{1}(\lambda)=\ext_{h(\lambda)}(F_{1})$,
$E^{2}(\lambda)=\ext_{h(\lambda)}(F_{2})$
and $E=a_{1}E^{1}+a_{2}E^{2}$ for simplicity.

Then,
the Levi form of $\rho$ is 
\begin{align*}
\levi{\rho}[v,\overline{v}]
&=
\frac{
c\left.E_{\lambda\overline{\lambda}}\right|_{\lambda=0}
}
{(c+E(0))^{3}}
+
\frac{
E(0)\left.E_{\lambda\overline{\lambda}}\right|_{\lambda=0}
-
2|\left.E_{\lambda}\right|_{\lambda=0}|^{2}
}
{(c+E(0))^{3}}.
\end{align*}
Therefore,
we conclude 
from Corollary \ref{coro:strong-positivity}
that
\begin{align*}
&E(0)\left.E_{\lambda\overline{\lambda}}\right|_{\lambda=0}
-
2|\left.E_{\lambda}\right|_{\lambda=0}|^{2} \\
&=
a_{1}^{2}\left(
E^{1}(0)\left.E^{1}_{\lambda\overline{\lambda}}\right|_{\lambda=0}
-
2|\left.E^{1}_{\lambda}\right|_{\lambda=0}|^{2}
\right)
+a_{2}^{2}\left(
E^{2}(0)\left.E^{2}_{\lambda\overline{\lambda}}\right|_{\lambda=0}
-
2|\left.E^{2}_{\lambda}\right|_{\lambda=0}|^{2}
\right) \\
&\quad+
a_{1}a_{2}\left(
E^{1}(0)\left.E^{2}_{\lambda\overline{\lambda}}\right|_{\lambda=0}+
E^{2}(0)\left.E^{1}_{\lambda\overline{\lambda}}\right|_{\lambda=0}
-
4{\rm Re}\left(
\overline{\left.E^{1}_{\lambda}\right|_{\lambda=0}}
\left.E^{2}_{\lambda}\right|_{\lambda=0}
\right)
\right) \\
&\ge
a_{1}a_{2}\left(
E^{1}(0)\left.E^{2}_{\lambda\overline{\lambda}}\right|_{\lambda=0}+
E^{2}(0)\left.E^{1}_{\lambda\overline{\lambda}}\right|_{\lambda=0}
-
4
\left|
\left.E^{1}_{\lambda}\right|_{\lambda=0}
\right|
\left|
\left.E^{2}_{\lambda}\right|_{\lambda=0}
\right|
\right) \\
&\ge
a_{1}a_{2}\left(
E^{1}(0)\left.E^{2}_{\lambda\overline{\lambda}}\right|_{\lambda=0}+
E^{2}(0)\left.E^{1}_{\lambda\overline{\lambda}}\right|_{\lambda=0}
\right.
\\
&
\left.\qquad-
2
E^{1}(0)^{1/2}
\left(
\left.E^{2}_{\lambda\overline{\lambda}}\right|_{\lambda=0}
\right)^{1/2}
E^{1}(0)^{1/2}
\left(\left.E^{2}_{\lambda\overline{\lambda}}\right|_{\lambda=0}
\right)^{1/2}
\right) \\
&=
a_{1}a_{2}\left(
E^{1}(0)^{1/2}
\left(\left.E^{2}_{\lambda\overline{\lambda}}\right|_{\lambda=0}
\right)^{1/2}
-
E^{1}(0)^{1/2}
\left(
\left.E^{2}_{\lambda\overline{\lambda}}\right|_{\lambda=0}
\right)^{1/2}
\right)^{2}
\ge 0
\end{align*}
and hence
$$
\levi{\rho}[v,\overline{v}]
\ge
\frac{
c\left.E_{\lambda\overline{\lambda}}\right|_{\lambda=0}
}
{(c+E(0))^{3}}
=
c\frac{
a_{1}^{2}\levi{\ext_{\cdot}(F_{1})}[v,\overline{v}]+a_{2}^{2}\levi{\ext_{\cdot}(F_{2})}[v,\overline{v}]
}
{(c+a_{1}\ext_{x_{0}}(F_{1})+a_{2}\ext_{x_{0}}(F_{2}))^{3}}
$$
when 
$a+b>0$, $c>0$.
Indeed,
since $q_{F,x}$ is generic,
the anti-complex linear map
$$
T_{x_{0}}\teich{g,m}\ni v\mapsto \eta_{v}\in \quaddiff{x_{0}}
$$
is an isomorphism.
Therefore,
$\rho$ is (real analytic) strictly plurisubharmonic on $\teich{g,m}$.
From the definition,
$\rho$ satisfies that
$\rho(x)<0$ for $x\in \teich{g,m}$.
Since $F_{1}$ and $F_{2}$ are transverse,
for any $\epsilon>0$,
there is a compact set $K\subset \teich{g,m}$
such that $\rho>-\epsilon$ for all $x\in \teich{g,m}-K$
(cf. \S\ref{Kerckhoff-Minsky}).
Since $\rho(x)>-1/a$ for all $x\in \teich{g,m}$,
the function $\rho(x)$
satisfies the desired condition.
\end{proof}

We also obtain the following comparizon.


\begin{theorem}[Gradients and Levi forms for log-extremal length functions]
\label{thm:strong-positivity2-2}
For any measured foliation
$F$,
\begin{align*}
d\log \ext_{\cdot}(F)
\wedge
d^{c}\log \ext_{\cdot}(F)
&\le 
\frac{1}{2\ext_{\cdot}(F)}
dd^{c}
\,
\ext_{\cdot}(F)
\le 
dd^{c}\log\ext_{\cdot}(F)
\end{align*}
holds
in the sense of currents,
where $d=\partial+\overline{\partial}$
and $d^{c}=i(\overline{\partial}-\partial)$.
Especially,
the log-extremal length function
$\log\ext_{\cdot}(F)$ is a real analytic strictly plurisubharmonic function on $\teich{g,m}$,
when $F$ is essentially complete.
\end{theorem}

\begin{proof}
By approximation,
we may assume that $F$ is essentially complete.
Then,
the inequalities follow
from the combinations of \eqref{eq:log-laplacian}
and
Corollary \ref{coro:strong-positivity}.
\end{proof}

\section{Gemetric interpretation of $\eta_{v}$}
\label{sec:geometric-interpretation}
\subsection{Motivation}
Let $x_{0}=(M_{0},f_{0})\in \teich{g,m}$.
Suppose $q_{0}=q_{F,x_{0}}$ is generic.
Let $v=[\dot{\mu}]\in T_{x_{0}}\teich{g,m}$.
We defined the quadratic differential $\eta_{v}\in \quaddiff{g,m}$
to satisfy
\begin{equation}
\label{eq:pairing-and-projection}
\langle v,\psi\rangle=\int_{M_{0}}\frac{\overline{\eta_{v}}}{|q_{0}|}\psi
\end{equation}
for all $\psi\in \quaddiff{g,m}$
in Lemma \ref{lem:transversal} and Theorem \ref{thm:leviform}.
The definition of $\eta_{v}$
implies
as if the infinitesimal Beltrami differential $\dot{\mu}$
were infinitesimally equivalent to the ``infinitesimal deformation'' along
$\overline{\eta_{v}}/|q_{0}|$
on $M_{q_{0}}$
(cf. \cite{MR1215481}).
However,
$\overline{\eta_{v}}/|q_{0}|$
is not a Beltrami differential in general,
becauce it may have $1/|z|$-singularities at zeroes of $q_{0}$
and hence it could happen $\|\overline{\eta_{v}}/|q_{0}|\|_{\infty}=\infty$.
Therefore,
the above discussion does not interpret correctly
the geometry of the anti-complex linear map
\begin{equation}
\label{eq:corresponding-v-eta}
T_{x_{0}}\teich{g,m}\ni v\mapsto \eta_{v}\in \quaddiff{x_{0}}
\end{equation}
 in general.
In this section,
we shall try to give a geometric description of the correspondence
\eqref{eq:corresponding-v-eta}.

\subsection{Geometric interpretation}
From \cite[Theorem 5.8]{2012arXiv1203.0273D} and \cite{MR1099913},
the map 
\begin{equation}
\label{eq:Gardiner-Masur}
\quaddiff{g,m}\ni q\mapsto (\horizontal(q),\vertical(q))\in \mathcal{MF}\times \mathcal{MF}
\end{equation}
is a real analytic diffeomorphism
around $q_{0}$.
%
Hence, the map
\begin{equation}
\label{eq:Gardiner-Masur2}
\complexstructure\colon
\teich{g,m}\times \teich{g,m}
\ni (x,y)\mapsto -q_{\horizontal(q_{F,x}),y}\in 
\quaddiff{g,m}
\end{equation}
is also a real analytic diffeomorphism on a neighborhood $N_{1}\times N_{1}$
of $(x_{0},x_{0})$
onto the image,
where $N_{1}$ is a neighborhood of $x_{0}=\pi(q_{0})\in \teich{g,m}$.
Notice that 
$\horizontal(\complexstructure(x,y))=\horizontal(q_{F,x})$,
$\complexstructure(x,y)\in \quaddiff{y}$
and $\complexstructure(x,x)=q_{F,x}$ for all $x,y\in N_{1}$.
Let $\complexstructure_{y}(x)=\complexstructure(x,y)$ for $(x,y)\in N_{1}\times N_{1}$
(Figure \ref{fig:mapJ}).
\begin{figure}
\includegraphics[height=6cm]{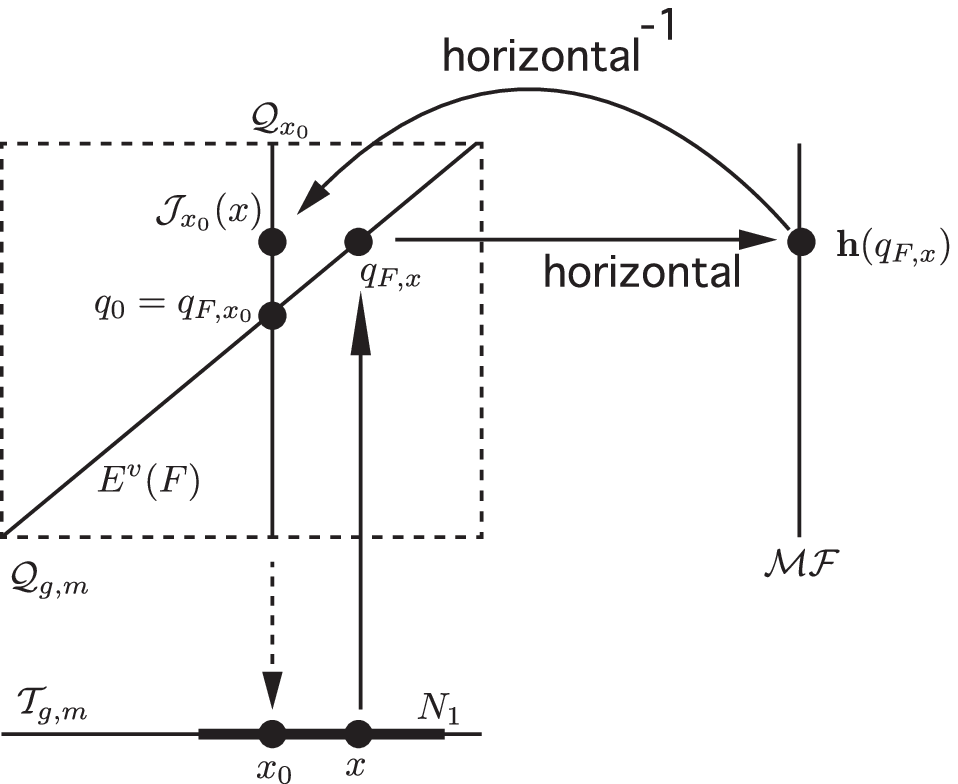}
\caption{A schematic of the map $\mathcal{J}_{x_{0}}$.}
\label{fig:mapJ}
\end{figure}


\begin{theorem}[Geometric interpretation]
\label{thm:duality}
For $v\in T_{x_{0}}\teich{g,m}$,
we have
$$
(\complexstructure_{x_{0}})_{*}[v]=-4\eta_{v},
$$
where $\eta_{v}\in \quaddiff{x_{0}}$ is taken
as \eqref{eq:pairing-and-projection}.
\end{theorem}

\begin{proof}
Consider the triangulation $\Delta_{q}$ and
the train track $\tau_{q_{0}}$ on $M_{q}$
for the horizontal foliation $|{\rm Im}(\sqrt{q})|$
of $q$
(cf. \cite[Theorem 5.8]{2012arXiv1203.0273D}).
We may assume that $\Delta_{q}$ contains no horizontal edge
(cf. \cite[Lemma 5.7]{2012arXiv1203.0273D}).
Consider the lifts $\tilde{\Delta}_{q}$ and $\tilde{\tau}_{q}$ of
$\Delta_{q}$ and $\tau_{q}$ to $\tilde{M}_{q}$,
respectivey.
We orient each edge of $\tilde{\Delta}_{q}$ so that
${\rm Im}(\omega_{q})$ is positive along the edge.
Each branch $b$ of $\tilde{\tau}_{q}$ intersects
a unique edge $e_{b}$ of $\tilde{\Delta}_{q}$.
We orient $b$ so that $b\cdot e_{b}=+1$,
where $\cdot$ means the algebraic intersection number.
By taking $U$ sufficiently small,
we may also assume that the marking $\tilde{M}_{q_{0}}\to \tilde{M}_{q}$
induces isomorphisms from $\tilde{\Delta}_{q_{0}}$ and
$\tilde{\tau}_{q_{0}}$
to $\tilde{\Delta}_{q}$ and $\tilde{\tau}_{q}$,
respectively.

From \cite[Theorem 5.8]{2012arXiv1203.0273D},
the tangential map of
the map $\horizontal\colon \quaddiff{x_{0}}\to \mathcal{MF}$
of the horizontal foliations
satisfies 
\begin{equation}
\label{eq:Dumas1}
{\rm Im}\int_{M_{q_{0}}}\frac{\psi_{1}\overline{\psi_{2}}}{4|q_{0}|}
=\ThurstonSymplectic{\horizontal_{*}(\psi_{1})}{\horizontal_{*}(\psi_{2})}{\tau_{q_{0}}}
\end{equation}
for $\psi_{1},\psi_{2}\in 
T_{q_{0}}\quaddiff{x_{0}}\cong \quaddiff{x_{0}}$,
where $\ThurstonSymplectic{\cdot}{\cdot}{\tau_{q_{0}}}$
is the Thurston's symplectic form on $W(\tau_{q_{0}})$,
where $W(\tau_{q_{0}})$ is the space
of real valued functions on the branches of $\tau_{q_{0}}$
satisfying the switch condition.
It is known that the space $W(\tau_{q_{0}})$
is identified with the tangent space of $\mathcal{MF}$
at $\horizontal(q_{0})$
since $\tau_{q_{0}}$ is maximal
(cf. \cite[\S3.2]{MR1144770}).

From the definition,
$|{\rm Im}(\omega_{q_{F,x}})|$ is the lift of the horizontal foliation
of $q_{F,x}$ to $\tilde{M}_{q_{F,x}}$.
By the definition of the orientation of edges of $\tilde{\Delta}_{q_{0}}$,
the image of the map
$N_{1}\ni x\mapsto \horizontal(q_{F,x})\in \mathcal{MF}$
at $x$ is identified with $w_{x}\in W(\tilde{\tau}_{q_{0}})$
which is defined by
$$
w_{x}(b)=\int_{b}|{\rm Im}(\omega_{q_{F,x}})|
=\int_{b}{\rm Im}(\omega_{q_{F,x}}).
$$
From Lemmas \ref{lem:Douady-Hubbard} and \ref{lem:transversal},
the image of $v\in T_{x_{0}}\teich{g,m}$ under the differential of the map
$N_{1}\ni x\mapsto {\rm Im}(\chi_{q_{F,x}})$
is identified with
the homomorphism
\begin{equation}
\label{eq:derivative1}
V_{v}
\colon H_{1}(\tilde{M}_{q_{0}})
\ni c\mapsto 
-2{\rm Im}\int_{c}
\frac{(\pi_{q_{0}})^{*}(\eta_{v})}{\omega_{q_{0}}}
\in \mathbb{R}.
\end{equation}
Indeed,
when $\{q_{\lambda}\}_{\lambda}\in D$
is taken as Lemma \ref{lem:transversal},
\eqref{eq:derivative1}
is derived from
\begin{align*}
\frac{\partial {\rm Im}(\chi_{q_{\lambda}}(c))}{\partial t}
&=
-i\left(\frac{\partial \chi_{q_{\lambda}}(c)}{\partial \lambda}
+\frac{\partial \chi_{q_{\lambda}}(c)}{\partial \overline{\lambda}}
\right)
\\
&=-i\int_{c}
\overline{
\left(
\frac{(\pi_{q_{\lambda}})^{*}(\eta_{\lambda})}{\omega_{q_{\lambda}}}
\right)}
-\frac{(\pi_{q_{\lambda}})^{*}(\eta_{\lambda})}{\omega_{q_{\lambda}}}
\\
&=-2{\rm Im}\int_{c}
\frac{(\pi_{q_{\lambda}})^{*}(\eta_{\lambda})}{\omega_{q_{\lambda}}}
\end{align*}
where $\lambda=t+is$,
since the real part of $\chi_{q_{\lambda}}$ is constant on $D$.

Let $p\colon \tilde{\tau}_{q_{0}}\to \tau_{q_{0}}$ be the
induced covering.
For a weight $w\in W(\tau_{q_{0}})$,
we define $C_{w}\in H_{1}(\tilde{M}_{q_{0}},\mathbb{R})^{-}$
as the homology class of a cycle
$$
\sum_{\mbox{$b$:branches of $\tilde{\tau}_{q_{0}}$}}w(p(b))b.
$$
Then,
\begin{align*}
w_{x}(C_{w})
&=
\sum_{b}w(p(b))w_{x}(b)
\\
V_{v}(C_{w})
&=
-2{\rm Im}\int_{C_{w}}
\frac{(\pi_{q_{0}})^{*}(\eta_{v})}{\omega_{q_{0}}}
=
-2
\sum_{b}w(p(b))
{\rm Im}\int_{b}
\frac{(\pi_{q_{0}})^{*}(\eta_{v})}{\omega_{q_{0}}}.
\end{align*}
These equations imply that
the differential of the map $N_{1}\ni x\mapsto w_{x}\in W(\tilde{\tau}_{q_{0}})$ at $x=x_{0}$ is equal to the weight
(which we abbreviate as)
$$
V_{v}(b)=-2\,{\rm Im}\int_{b}
\frac{(\pi_{q_{0}})^{*}(\eta_{v})}{\omega_{q_{0}}}.
$$
The weight $V_{v}$ is invariant under the action of covering
transformation $r_{0}$ of $\tilde{M}_{q_{0}}\to M_{q_{0}}$
on $\tilde{\tau}_{q_{0}}$,
because
for any branch $b$ of $\tilde{\tau}_{q_{0}}$,
the orientation of $r_{0}(b)$ is opposite to the orientation induced
from $b$.
Thus,
$V_{v}$ descends to a weight $p_{*}(V_{v})$
on $\tau_{q_{0}}$.

Let $\psi_{0}=(\complexstructure_{x_{0}})_{*}[v]$.
For all $\psi\in \quaddiff{x_{0}}\subset T_{q_{0}}\quaddiff{g,m}$,
we define $w(\psi)\in W(\tilde{\tau}_{q_{0}})$ by
$$
w(\psi)(b)={\rm Im}\int_{b}\frac{(\pi_{q_{0}})^{*}(\psi)}{2\omega_{q_{0}}}.
$$
Then,
we have
$\horizontal_{*}[\psi]=p_{*}(w(\psi))$
(cf. \cite[Proof of Theorem 5.8]{2012arXiv1203.0273D}).
By the definition of $\complexstructure_{x_{0}}$,
$\horizontal\circ \complexstructure_{x_{0}}(x)=\horizontal(q_{F,x})$
for all $x\in N_{1}$.
Hence,
$\horizontal_{*}[\psi_{0}]=p_{*}(V_{v})$
as elements of $W(\tilde{\tau}_{q_{0}})
\cong T_{\horizontal(q_{0})}\mathcal{MF}$.
From \eqref{eq:Dumas1}
and \cite[(5.4)]{2012arXiv1203.0273D},
the following holds for all $\psi\in \quaddiff{x_{0}}$:
\begin{align*}
{\rm Im}\int_{M_{q_{0}}}\frac{\psi_{0}\overline{\psi}}{4|q_{0}|}
&=\frac{1}{2}\,{\rm Im}\int_{\tilde{M}_{q_{0}}}
\frac{(\pi_{q_{0}})^{*}(\psi_{0})\overline{(\pi_{q_{0}})^{*}(\psi)}}
{4|\omega_{q_{0}}|^{2}}
\\
&=
\frac{1}{2}\,\ThurstonSymplectic{V_{v}}{w(\psi)}{\tilde{\tau}_{q_{0}}} \\
&=
-2\int_{\tilde{M}_{q_{0}}}
{\rm Im}\frac{(\pi_{q_{0}})^{*}(\eta_{v})}{2\omega_{q_{0}}}
\wedge
{\rm Im}\frac{(\pi_{q_{0}})^{*}(\psi)}{2\omega_{q_{0}}} \\
&=
-{\rm Re}\int_{\tilde{M}_{q_{0}}}
\frac{(\pi_{q_{0}})^{*}(\eta_{v})}{2\omega_{q_{0}}}
\wedge
\overline{\frac{(\pi_{q_{0}})^{*}(\psi)}{2\omega_{q_{0}}}}
\\
&=
-2{\rm Im}\int_{\tilde{M}_{q_{0}}}
\frac{(\pi_{q_{0}})^{*}(\eta_{v})\overline{(\pi_{q_{0}})^{*}(\psi)}}
{4|\omega_{q_{0}}|^{2}}
=
-{\rm Im}\int_{M_{q_{0}}}\frac{\eta_{v}\overline{\psi}}{|q_{0}|}.
\end{align*}
Therefore,
we have
$(\complexstructure_{x_{0}})_{*}[v]
=\psi_{0}=-4\eta_{v}$.
\end{proof}

\section{Appendix}

\subsection{Alternative approach to Gardiner's formula}
\label{subsec:Appendix-Gardiner-formula}
Applying our method in Theorem \ref{thm:leviform},
one can also get an alternative proof of
Gardiner's formula for the derivative of 
extremal length in some case.
Indeed,
this is the case where we need in the proof of Theorem \ref{thm:pluri-subharmonicity}.

Let $F\in \mathcal{MF}$ and $x_{0}=(M,f)\in \teich{g,m}$.
Suppose that the Hubbard-Masur differential $q_{0}=q_{F,x_{0}}$
is generic.
Let $v=[\dot{\mu}]$ and $\{q_{\lambda}\}_{\lambda\in D}$
as the proof of Theorem \ref{thm:leviform}.
Suppose $x_{\lambda}=\pi(q_{\lambda})$ varies holomorphically.
By applying the discussion in the proof of
Theorem \ref{thm:leviform},
from Lemma \ref{lem:transversal},
we have
\begin{align*}
&\left.
\frac{\partial}{\partial\lambda}
\ext_{x_{\lambda}}(F)
\right|_{\lambda=0}\\
&=
\frac{i}{4}
\sum_{k=1}^{\tilde{g}}
\left(
\int_{\alpha_{k}}\overline{
\left(
\frac{(\pi_{q_{0}})^{*}(\eta_{v})}{\omega_{q_{0}}}
\right)
}\int_{\beta_{k}}\overline{\omega_{q_{0}}}
-
\int_{\beta_{k}}\overline{
\left(
\frac{(\pi_{q_{0}})^{*}(\eta_{v})}{\omega_{q_{0}}}
\right)
}
\int_{\alpha_{k}}\overline{\omega_{q_{0}}}
\right) \\
&\quad
+
\frac{i}{4}
\sum_{k=1}^{\tilde{g}}
\left(
\int_{\alpha_{k}}\omega_{q_{0}}
\int_{\beta_{k}}
\overline{
\left(
-\frac{(\pi_{q_{0}})^{*}(\eta_{v})}{\omega_{q_{0}}}
\right)
}
-
\int_{\beta_{k}}\omega_{q_{0}}
\int_{\alpha_{k}}\overline{
\left(-
\frac{(\pi_{q_{0}})^{*}(\eta_{v})}{\omega_{q_{0}}}
\right)
}
\right) \\
&=
\frac{i}{4}
\int_{\tilde{M}_{q_{0}}}
\overline{
\left(
\frac{(\pi_{q_{0}})^{*}(\eta_{v})}{\omega_{q_{0}}}
\right)}
\wedge \overline{\omega_{q_{0}}}
-
\frac{i}{4}
\int_{\tilde{M}_{q_{0}}}\omega_{q_{0}}
\wedge 
\overline{
\left(
\frac{(\pi_{q_{0}})^{*}(\eta_{v})}{\omega_{q_{0}}}
\right)}
\\
&=-\frac{1}{2}
\int_{\tilde{M}_{q_{0}}}
\frac{\overline{(\pi_{q_{0}})^{*}(\eta_{v})}}{|\omega_{q_{0}}|^{2}}
\omega_{q_{0}}^{2}
=
-\int_{M_{q_{0}}}
\frac{\overline{\eta_{v}}}{|q_{0}|}q_{0}
=
-\int_{M}
\dot{\mu}\,q_{F,x_{0}}
\end{align*}
because $M=M_{q_{0}}$ by definition.

\subsection{Examples}
\label{subsec:examples}
In this section,
we treat the exceptional cases
in the proofs of pluriharmonicity in the previous sections.
\subsubsection{Levi forms for the case of flat tori}
\label{subsubsec:tori}
We start to consider
the Levi forms of extremal length functions on the Teichm\"uller space
of flat tori (once punctured tori),
and check our formula in Theorem \ref{thm:leviform} also valid in this case.
This case is excluded from our assumption in the theorem.
However,
this is a motivated example of our result.

The Teichm\"uller space of a flat torus is identified with the upper half plane $\mathbb{H}$.
A point $\tau\in \mathbb{H}$ corresponds to a marked Riemann surface
$M_{\tau}$
which is defined as the quotient space of $\mathbb{C}$
by a lattice generated by $z+1$ and $z+\tau$.
Notice that the space $\mathcal{PMF}=\mathcal{PMF}(M_{i})$
of projective measured foliations on $M_{i}$
($i=\sqrt{-1}\in \mathbb{H}$)
is identified with the projective space $\mathbb{RP}^{1}=\{[x\colon y]\mid
(x,y)\ne (0,0)\}$.
A point $[x\colon y]\in \mathbb{RP}^{1}$
corresponds to
a measured foliation whose leaves are parallel to the line of direction $x+yi$.

Let $\alpha$ be a closed curve on $M_{i}$
corresponding to $[1\colon 0]\in \mathcal{PMF}$.
Then,
the Jenkins-Strebel differential $q_{\alpha,\tau}$ on $M_{\tau}$
for $\alpha$
and
the extremal length function of $\alpha$ at $M_{\tau}$
are obtained as
\begin{align*}
q_{\alpha,\tau}&=-\frac{1}{{\rm Im}(\tau)^{2}}dz^{2}\\
\ext_{\tau}(\alpha)
&=\frac{1}{{\rm Im}(\tau)}.
\end{align*}
Fix $\tau_{0}\in \mathbb{H}$.
For $V\in \mathbb{C}$,
let $v$ be the tangent vector at $\tau_{0}$
corresponding to the infinitesimal quasiconformal
deformation
from $M_{\tau_{0}}$ to $M_{\tau_{0}+\lambda V}$ as $\lambda\to 0$.
The Levi form (Laplacian) of the extremal length function is
\begin{equation}
\label{eq:laplacian-extremal-length-function1}
\levi{\ext_{\cdot}(\alpha)}[v,\overline{v}]=
\left.
\frac{\partial^{2}}{\partial \lambda\partial\overline{\lambda}}
\ext_{\tau_{0}+\lambda V}(\lambda)
\right|_{\lambda=0}
=\frac{|V|^{2}}{2{\rm Im}(\tau_{0})^{3}}
\end{equation}
The Beltrami differential $\mu(\lambda)$
of the quasiconformal deformation
from $M_{\tau_{0}}$ to $M_{\tau_{0}+\lambda V}$
behaves
$$
\mu(\lambda)=\lambda\cdot \dot{\mu}+o(|\lambda|)=
\lambda\cdot \frac{iV}{2{\rm Im}(\tau_{0})}\frac{d\overline{z}}{dz}+o(|\lambda|)
$$
as $t\to 0$.
By definition,
$v=[\dot{\mu}]$.
We define
a holomorphic quadratic differential $\eta_{v}$ on $M_{\tau_{0}}$ by 
\begin{equation}
\label{eq:laplacian-extremal-length-function3}
\eta_{v}=-\frac{i\overline{V}}{2{\rm Im}(\tau_{0})^{3}}dz^{2}.
\end{equation}
This differential satisfies that
$$
\langle v,\psi\rangle=
\int_{M_{\tau_{0}}}\dot{\mu}\psi=
\int_{M_{\tau_{0}}}\frac{\overline{\eta_{v}}}{|q_{\alpha,\tau_{0}}|}\psi
$$
for all holomorphic quadratic differential $\psi$ on $M_{\tau_{0}}$.
Therefore,
the right-hand side of the formula in Theorem \ref{thm:leviform}
is equal to
$$
2\int_{M_{\tau_{0}}}\frac{|\eta_{v}|^{2}}{|q_{\alpha,\tau_{0}}|}
=\frac{|V|^{2}}{2{\rm Im}(\tau_{0})^{3}},
$$
which coincides with the right-hand side of \eqref{eq:laplacian-extremal-length-function1}.

Let $F$ be a measured foliation
corresponding to $[a\colon b]\in \mathcal{PMF}$ with $b\ne 0$.
We normalize $F$ with $i(\alpha,F)=1$.
Then the Hubbard-Masur differential $q_{F,\tau}$ on $M_{\tau}$
and the extremal length function for $F$
are
\begin{align*}
q_{F,\tau}
&=-\frac{(a+b\overline{\tau})^{2}}{b^{2}{\rm Im}(\tau)^{2}}dz^{2}
\\
\ext_{\tau}(F)
&=
\frac{|a+b\tau|^{2}}{b^{2}{\rm Im}(\tau)}.
\end{align*}
Let $v$ be the tangent vector at $\tau_{0}$
corresponding to
$\dot{\mu}=\frac{iV}{2{\rm Im}(\tau_{0})}\frac{d\overline{z}}{dz}$.
Then,
\begin{equation}
\label{eq:laplacian-extremal-length-function2}
\levi{\ext_{\cdot}(\alpha)}[v,\overline{v}]=
\left.
\frac{\partial^{2}}{\partial \lambda\partial\overline{\lambda}}
\ext_{\tau_{0}+\lambda V}(\lambda)
\right|_{\lambda=0}
=\frac{|a+b\tau_{0}|^{2}}{2b^{2}{\rm Im}(\tau_{0})^{3}}|V|^{2}
\end{equation}
In this case,
the differential $\eta_{v}\in \quaddiff{\tau_{0}}$
is
\begin{equation}
\label{eq:laplacian-extremal-length-function4}
\eta_{v}=-\frac{i|a+b\tau_{0}|^{2}
\overline{V}}{2b^{2}{\rm Im}(\tau_{0})^{3}}dz^{2}.
\end{equation}
and hence
$$
2\int_{M_{\tau_{0}}}\frac{|\eta_{v}|^{2}}{|q_{F,\tau_{0}}|}
=\frac{|a+b\tau_{0}|^{2}}{2b^{2}{\rm Im}(\tau_{0})^{3}}|V|^{2},
$$
which coincides with the right-hand side of \eqref{eq:laplacian-extremal-length-function2}.

\subsubsection{The gemetric interpretation of $\eta_{v}$ for the case of tori}
\label{subsubsec:example-tori}
We check Theorem \ref{thm:duality} in the case of flat tori
(once punctured tori).
We use the notation defined in \S\ref{subsubsec:tori} frequently.

First we consider the case $\alpha=[1\colon 0]$.
Fix $\tau_{0}\in \mathbb{H}$.
Since
the projective class of the horizontal foliation of $q_{\alpha,\tau}$ corresponds to
a projective class $[-{\rm Re}(\tau)\colon 1]\in \mathbb{RP}^{1}$,
the underlying foliation of
the horizontal foliation $J_{\tau_{0}}(\tau)$
is foliated by
the lines (leaves) directed to $-({\rm Re}(\tau))+\tau_{0}$.
Since the horizontal foliations of $J_{\tau_{0}}(\tau)$ and $q_{\alpha,\tau}$
are same,
$$
\int_{\alpha}|{\rm Im}\sqrt{J_{\tau_{0}}(\tau)}|
=i(\horizontal(J_{\tau_{0}}(\tau)),\alpha)
=i(\horizontal(q_{\alpha,\tau}),\alpha)
=
\int_{\alpha}|{\rm Im}\sqrt{q_{\alpha,\tau}}|
$$
holds.
Therefore,
we can check that
$$
J_{\tau_{0}}(\tau)=
\left(
\frac{-{\rm Re}(\tau)+\overline{\tau_{0}}}
{{\rm Im}(\tau){\rm Im}(\tau_{0})}
\right)^{2}
dz^{2}
\in \quaddiff{\tau_{0}}
$$
for $\tau\in \mathbb{H}$.
Let $v=[\dot{\mu}]$ be a tangent vector at $\tau_{0}$,
where
$\dot{\mu}=\frac{iV}{2{\rm Im}(\tau_{0})}\frac{d\overline{z}}{dz}$
with $V\in \mathbb{C}$
as \S\ref{subsubsec:tori}.
Then,
from the above calculation,
we obtain
$$
(J_{\tau_{0}})_{*}[v]=
\left(2\frac{i\overline{V}}{{\rm Im}(\tau_{0})^{3}}
\right)dz^{2}
\in \quaddiff{\tau_{0}}.
$$
Comparing with \eqref{eq:laplacian-extremal-length-function3},
we get $(J_{\tau_{0}})_{*}[v]=-4\eta_{v}$
as we appeared in Theorem \ref{thm:duality}.

Let $F\in \mathcal{PMF}$ be a measured foliation
corresponding to $[a\colon b]\in \mathcal{PMF}$ with $b\ne 0$
and $i(\alpha,F)=1$.
Then,
the underlying foliation of the horizontal foliation of $q_{F,\tau}$
is foliated by the lines of direction $[a{\rm Re}(\tau)+b|\tau|^{2}\colon a+b{\rm Re}(\tau)]$.
By the similar argument as above,
one can check that
\begin{align*}
J_{\tau_{0}}(\tau)
&=\left(
\frac{-(a{\rm Re}(\tau)+b|\tau|^{2})+(a+b{\rm Re}(\tau))\overline{\tau_{0}}}
{b{\rm Im}(\tau){\rm Im}(\tau_{0})}
\right)^{2}
dz^{2}
\\
(J_{\tau_{0}})_{*}[v]
&=
\left(2\frac{i|a+b\tau_{0}|^{2}\overline{V}}{b^{2}{\rm Im}(\tau_{0})^{3}}
\right)dz^{2}
\in \quaddiff{\tau_{0}}.
\end{align*}
%
From \eqref{eq:laplacian-extremal-length-function4},
we also have
$(J_{\tau_{0}})_{*}[v]=-4\eta_{v}$.

\subsubsection{A comment on the case of $3g-3+m=1$}
The case of once punctured tori is treated in the same way
as that in the case of flat tori.
Indeed,
the canonical completion at the puncture gives an isomorphism between
the Teichm\"uller space of once punctured tori and that of flat tori.
The extremal length functions in the both spaces
are identified under the isomorphism.
The case of fours punctured sphere
is treated in the similar way.


\bibliographystyle{plain}
\bibliography{References-2}

\end{document}